\documentclass[11pt]{amsart}
\usepackage{amsmath,amssymb,amsthm,graphicx}

\usepackage[left=1in,right=1in,top=1in, bottom=1in]{geometry}
\usepackage{babel}
\usepackage[utf8]{inputenc}
\usepackage{geometry}
\usepackage{amsmath}
\usepackage{bm}
\usepackage{bbm}
\usepackage{mathtools} 
\usepackage{mathrsfs}
\usepackage{amsfonts}
\usepackage{amsthm}
\usepackage{enumitem}
\usepackage{amssymb}
\usepackage{esint}
\usepackage{array}
\usepackage{makecell}
\usepackage{enumitem}  
\usepackage{tabto}
\usepackage{fontawesome5}
\usepackage{pgfplots}
\pgfplotsset{compat=1.18}
\usepackage{comment}
\usepackage[square,numbers]{natbib}
\usepackage{hyperref}

\DeclareMathOperator{\R}{\mathbb{R}}
\DeclareMathOperator{\N}{\mathbb{N}}

\DeclareMathOperator{\E}{\mathbb{E}}
\DeclareMathOperator{\Sp}{\mathbb{S}}
\DeclareMathOperator{\Pb}{\mathbb{P}}
\DeclareMathOperator{\I}{\mathbbm{1}}
\DeclareMathOperator{\TZ}{\mathcal{Z}}
\DeclareMathOperator{\TL}{\mathcal{L}}

\DeclareMathOperator{\TE}{\mathcal{T}_{\cE}}

\DeclareMathOperator{\TX}{\mathcal{X}}
\DeclareMathOperator{\TY}{\mathcal{Y}}
\DeclareMathOperator{\cC}{\mathcal{C}}

\DeclareMathOperator{\cN}{\mathcal{N}}
\DeclareMathOperator{\cL}{\mathcal{L}}
\DeclareMathOperator{\cM}{\mathcal{M}}
\DeclareMathOperator{\cE}{\mathcal{E}}
\DeclareMathOperator{\cD}{\mathcal{D}}
\DeclareMathOperator{\cH}{\mathcal{H}}

\DeclareMathOperator{\cF}{\mathcal{F}}

\DeclareMathOperator{\cG}{\mathcal{G}}
\DeclareMathOperator{\cT}{\mathcal{T}}
\DeclareMathOperator{\taue}{\textit{t}_{\mathcal{E}}}

\DeclareMathOperator{\trace}{\text{tr}}
\DeclareMathOperator{\tra}{\top}

\newcommand{\parent}[1]{\left(#1\right)}
\newcommand{\norm}[1]{\left\lVert#1\right\rVert}
\newcommand{\inn}[1]{\langle#1\rangle}
\newcommand{\graf}[1]{\left\{#1\right\}}
\newcommand{\abs}[1]{\left\lvert#1\right\rvert}
\newcommand{\limn}{\lim_{n\to\infty}}

\newtheorem{thm}{Theorem}[section] 
\newtheorem{lemma}[thm]{Lemma}
\newtheorem{defi}[thm]{Definition}
\newtheorem{prop}[thm]{Proposition}
\newtheorem{cor}[thm]{Corollary}

\theoremstyle{remark}
\newtheorem{rem}[thm]{Remark}
\newtheorem{ex}[thm]{Example}
\numberwithin{equation}{section}

\makeatletter
\def\namedlabel#1#2{\begingroup  
    (#2)%
    \def\@currentlabel{#2}%
    \phantomsection\label{#1}\endgroup
}
\makeatother

\title{Time-changed reflected diffusions: subdiffusion, superdiffusion and explosion}



\author{Aleksandar Mijatovi\'c}
\address{Department of Statistics, University of Warwick, UK}
\email{a.mijatovic@warwick.ac.uk}
\author{Francesco Virgili}
\address{Department of Mathematics “Giuseppe Peano”, University of Turin, Turin, Italy}
\email{francesco.virgili@unito.it}

\begin{document}

\begin{abstract}
Motivated by semi-Markov models of anomalous diffusion, we study
obliquely reflected diffusions in generalized parabolic domains
under independent inverse-subordinator time changes.
We focus on infinite-mean subordinators, whose inverse clocks
can slow the superdiffusive growth induced by reflection.
For domains with cross-sectional radius asymptotic to
$a_\infty x^\beta$, where $-1<\beta<1$, we strengthen the
almost-sure growth laws of \cite{Mijatovic1} to obtain precise
second-moment asymptotics when the subordinator's Laplace
exponent is regularly varying at zero, and matching bounds
more generally. Inverse $\alpha$-stable clocks yield
mean-square growth of order $t^{2\alpha/(1+\beta)}$, giving
subdiffusive, diffusive, or superdiffusive regimes.
For $\beta<-1$, explosion persists almost surely, with finite
mean explosion time if and only if the subordinator has finite
mean. Supporting results establish strong existence, pathwise
uniqueness, and a time-change representation, also allowing
dependent clocks.
\end{abstract}

\date{\today}

\maketitle

\section{Introduction}

Inverse-subordinator time changes are widely used in semi-Markov
models to describe trapping and anomalous diffusion; see
\cite{biočić2026modellinganomalousdiffusionrole} and the references
therein. We study how this mechanism interacts with the
superdiffusive motion generated by oblique reflection in
generalized parabolic domains. 
\cite{Mijatovic1} establishes almost-sure growth laws for these
reflected diffusions and identifies a transition to finite-time
explosion when the domain narrows sufficiently rapidly.

Our main results describe the second-moment asymptotics after a time change via and
an independent inverse-subordinator. Particular
interest lies in infinite-mean subordinators, for which trapping
can change the growth exponent. For inverse stable clocks,
the interaction between the clock and the domain geometry
produces subdiffusive, diffusive, or superdiffusive mean-square
growth. We also determine how the time change affects explosion
and its expected time.

Let $\cD$ be a possibly unbounded closed $\cC^3$-domain in $\R^d$, and $\phi:\partial\cD\to\R^d$ a bounded $\cC^2$ vector field on the boundary of $\cD$, satisfying $\inf_{z\in\partial\cD}\inn{\phi(z),n(z)}>0$ where $n:\partial\cD\to\R^d$ is the inward normal vector field. Moreover assume that the maps $\mu:\cD\to\R^d$ and $\Sigma:\cD\to\cM^d_+$ 
are bounded and Lipschitz, and that the positive definite matrix $\Sigma$ is uniformly elliptic on $\cD$.  Denote by $\sigma:\cD\to\cM^d_+$ the positive definite square root of $\Sigma=
\sigma \sigma$.

Let a filtered probability space $(\Omega,\cF,\Pb)$ with filtration $\cH=(\cH_t)_{t\geq0}$ support a standard $\cH$-Brownian motion $W=(W_t)_{t\geq0}$ in $\R^d$, a continuous $\cH$-time change $T=(T_t)_{t\geq0}$
such that $T_t\to\infty$ a.s. as $t\to\infty$. Then the time-changed stochastic differential equation with reflection has a solution:
\begin{equation}\label{eq:TCRSDE}
    \TZ_{t}= z+\int_0^t \mu(\TZ_{s})dT_s+\int_0^t\sigma(\TZ_{s})dW_{T_s}+\int_0^t\phi(\TZ_s)d\TL_s,\quad\TL_t=\int_0^t\I(\TZ_s\in\partial\cD)d\TL_s.
\end{equation}
\begin{thm}\label{thm:existence and uniqueness}
    Let $\cD$ and $\phi$ satisfy \eqref{eq:domain exhaustion1}, then for any $z\in\cD$ there exists a strong solution $(\TZ,\TL,\TE)$ to \eqref{eq:TCRSDE} on the stochastic interval $[0,\TE)$, where $\TE$ is a possibly infinite stopping time, and pathwise uniqueness holds. Moreover
    \begin{equation}\label{eq:limit}
        \lim_{t\to\TE}\norm{\TZ_t}_{d}=\lim_{t\to\TE}\TL_t=\infty,\quad\text{on }\graf{\TE<\infty}.
    \end{equation}
\end{thm}

\subsection{Main result}
Our main result concerns the interaction between trapping and
reflection-induced superdiffusivity. Retain the coefficient
regularity above and the exhaustion assumption
\eqref{eq:domain exhaustion1}, set the interior drift in
\eqref{eq:TCRSDE} to zero, and consider $\TZ=(\TX,\TY)$ in
\begin{equation}\label{eq:domain-def}
\cD=\{z=(x,y)\in\R_+\times\R^d:\|y\|_d\leq b(x)\},
\qquad b(x)=a_\infty x^\beta+f(x)
\quad\text{for large }x,
\end{equation}
where $a_\infty>0$, $\beta<1$, and $f$ satisfies
\eqref{eq:f growth}. Assume
$b\in\cC([0,\infty))\cap\cC^3((0,\infty))$,
$b(0)=0$, $b(x)>0$ for $x>0$, and $\cD$ has $\cC^3$
boundary, including at the origin.

For $u\in\Sp^{d-1}$, write
$e_x=(1,0)$, $e_u=(0,u)$ and $\phi_x(u)=\phi(x,b(x)u)$.
Assume, for constants $s_0,c_0,\sigma^2>0$, that the reflection vector field satisfies
\begin{equation}\label{eq:e1-projection}
\sup_{u\in\Sp^{d-1}}
\left(
|\langle\phi_x(u),e_x\rangle-s_0|
+
|\langle\phi_x(u),-e_u\rangle-c_0|
\right)\longrightarrow0,\quad\text{as $x\to\infty$,}
\end{equation}
making it asymptotically oblique, and
\begin{equation}\label{eq:Sigma-limit}
\sup_{y:(x,y)\in\cD}
\left|\trace
\Sigma(x,y)-e_x^{\tra}\Sigma(x,y)e_x-\sigma^2\right|
\longrightarrow0,\quad\text{as $x\to\infty$.}
\end{equation}

Let $T_t\coloneqq \inf\{u\geq0:S_u>t\}$, where $S$ is an unkilled,
strictly increasing, driftless subordinator, independent
of the Brownian driver, with Laplace exponent $\varphi$.

\begin{thm}[Explosion and second-moment asymptotics]
\label{thm:indep subordinator}
Assume the standing regularity and domain-exhaustion hypotheses,
together with \eqref{eq:domain-def}, \eqref{eq:e1-projection},
\eqref{eq:Sigma-limit}, and \eqref{eq:f growth}.
Let $S$ be an unkilled, strictly increasing, driftless subordinator,
independent of $W$, with Laplace exponent $\varphi$, and recall
$T_t=\inf\{u\geq0:S_u>t\}$.
Then the following holds for every $z\in\cD$.

\begin{enumerate}[label=(\roman*)]
\item
The solution $\TZ$  of~\eqref{eq:TCRSDE} with zero interior drift $\mu\equiv0$ explodes almost surely if $\beta<-1$, and is
non-explosive if $\beta\geq-1$.
In the explosive case,
\[
\E_z\TE=\E[S_1]\,\E_z\taue,
\]
where $\taue$ is the explosion time of the corresponding
non-time-changed diffusion and $0<\E_z\taue<\infty$.
In particular, $\E_z\TE<\infty$ if and only if $\E[S_1]<\infty$.

\item
Suppose $-1<\beta<1$ and $\varphi$ is regularly varying
at zero with index $\alpha\in[0,1]$. Define
$p:=\frac{2}{1+\beta}$ and
$\kappa:=\frac{s_0\sigma^2(1+\beta)}{2c_0a_\infty}$.
Then $\TZ$  in~\eqref{eq:TCRSDE} with zero interior drift $\mu\equiv0$ satisfies 
\begin{equation}\label{eq:RValpha}
\E_z\norm{\TZ_t}_{d+1}^2
\sim
\kappa^p
\frac{\Gamma(1+p)}{\Gamma(1+\alpha p)}
\varphi(1/t)^{-p}\quad\text{as $t\to\infty$.}
\end{equation}
\end{enumerate}
\end{thm}

Equation \eqref{eq:RValpha} shows how choosing a subordinator in the domain of attraction of an $\alpha$-stable process leads to different diffusivity regimes: if $\varphi(\lambda)\sim
c\lambda^\alpha\ell(\lambda)$ as $\lambda\downarrow0$, Theorem~\ref{thm:indep subordinator} implies 
\[
\E_z\norm{\TZ_t}_{d+1}^2
\sim
\kappa^p\frac{\Gamma(1+p)}{\Gamma(1+\alpha p)}
t^{\alpha p}\bigl(c\ell(1/t)\bigr)^{-p}\quad\text{as $t\to\infty$.}
\]
\begin{ex}
    Specifically, assume $S$ is an $\alpha$-stable subordinator, i.e. its Laplace exponent is $\varphi(t)=t^{\alpha}$ for $\alpha\in(0,1)$. Assume we are in the regime $\beta\in(-1,1)$. Then~\eqref{eq:RValpha} becomes
    $$\E_z\norm{\TZ_t}_{d+1}^2\sim K_0t^{\frac{2\alpha}{1+\beta}}\quad t\to\infty,\qquad\text{where } K_0=\frac{\Gamma\parent{\frac{3+\beta}{1+\beta}}\parent{s_0\sigma^2(\beta+1)}^{\frac{2}{1+\beta}}}{\Gamma\parent{1+\frac{2\alpha}{1+\beta}}\parent{2c_0a_{\infty}}^{\frac{2}{1+\beta}}}.$$
    In contrast to the case of linear time change $T_t=\mu t$, where in the regime $\beta\in(-1,1)$ the process $\TZ$ is always super-diffusive~\cite{Mijatovic1}, when $T$ is an inverse of an  $\alpha$-stable subordinator we have:
\begin{center}
\begin{tabular}{c|c}
Range of $\alpha$ in $(0,1)$ & Mean-square behaviour ($\E_z\norm{\TZ_t}_{d+1}^2$ as $t\to\infty$) \\ \hline
$\frac{1+\beta}{2}<\alpha<1$ & Superdiffusive \\[4pt]
$\alpha=\frac{1+\beta}{2}$   & Diffusive \\[4pt]
$0<\alpha<\frac{1+\beta}{2}$ & Subdiffusive
\end{tabular}
\end{center}
\end{ex}
Without regular variation of $\varphi$, we retain the following
two-sided estimate.
\begin{prop}\label{prop:indep subordinator}
    In the same setting of Theorem \ref{thm:indep subordinator}, for $\beta\in(-1,1)$, and any Laplace exponent $\varphi$, we get that for any $z\in\cD$\footnote{Notice that we write $f(t)\asymp g(t)$ if there exists two constants $k_1,k_2>0$ such that for $t$ sufficiently big, it holds $$k_1f(t)\leq g(t)\leq k_2 f(t).$$}
    \begin{equation}
        \E_z\norm{\TZ_t}_{d+1}^2\asymp \varphi(1/t)^{-\frac{2}{1+\beta}} \text{ as }t\uparrow\infty.
    \end{equation}
\end{prop}
\begin{rem}

The case $\beta=-1$ is non-explosive, but is not covered
by our second-moment results. Extending these results
to the critical case requires additional moment estimates
and is left for future work.
\end{rem}

 Theorem~\ref{thm:indep subordinator} first establishes that the solution to \eqref{eq:TCRSDE} can explode if and only if $\beta<-1$, exactly as in \cite[Thm~2.2]{Mijatovic1}. However, in our model, a new scenario naturally arises. Since the time-change can delay the reflected process, when explosion occurs, the expected explosion time might be infinite, contrary to what happens in \cite[Theorem 2.2]{Mijatovic1}. For the non-explosive case, we can study how the asymptotic behavior is modified by the time-change.

For $-1<\beta<1$, the almost-sure growth law in
\cite[Thm~2.2]{Mijatovic1} and
Corollary~\ref{cor:rep of sol} imply
\[
\frac{\norm{\TZ_t}_{d+1}}{T_t^{1/(1+\beta)}}=\frac{\norm{Z_{T_t}}_{d+1}}{T_t^{1/(1+\beta)}}
\longrightarrow
\left(\frac{s_0\sigma^2(1+\beta)}
{2c_0a_\infty}\right)^{1/(1+\beta)}
\quad\text{a.s.}
\]
This holds whenever $T_t\to\infty$ almost surely, without
requiring independence. If $m:=\E[S_1]<\infty$, then
$T_t/t\to m^{-1}$ almost surely (see \cite{Bertoin1999}), so
\[
t^{-1/(1+\beta)}\norm{Z_{T_t}}_{d+1}
\longrightarrow
\left(\frac{s_0\sigma^2(1+\beta)}
{2c_0a_\infty m}\right)^{1/(1+\beta)}
\quad\text{a.s.}
\]
Thus finite-mean clocks preserve the growth exponent.
Infinite-mean subordinators can change the growth regime.
Motivated by their role in semi-Markov models, we study
second-moment asymptotics in physical time.
The almost-sure growth law alone does not justify averaging
over the independent random clock. Proposition
\ref{prop:L2 superdiffusivity} supplies the required moment
control by strengthening the underlying growth law to
$L^2$ convergence.

The preceding asymptotic results rely on a time-change
representation that holds in a more general setting.
We establish strong existence and pathwise uniqueness up
to explosion in regular unbounded domains, allowing an
interior drift and dependence between the clock and the
driving Brownian motion.

\subsection{Existence and uniqueness of reflected diffusions in unbounded domains}
Asymptotic results in Theorem~\ref{thm:indep subordinator}  rely on a time-change
representation 
given in Theorem~\ref{thm:existence and uniqueness} above,
establishing strong existence and pathwise uniqueness up
to explosion in regular unbounded domains, allowing an
interior drift and dependence between the clock and the
driving Brownian motion.

Theorem~\ref{thm:existence and uniqueness} also shows that
reflection commutes with the time change: both the reflected
diffusion and its boundary regulator are obtained by applying
the clock to their non-time-changed counterparts.

\begin{cor}\label{cor:rep of sol}
Under the assumptions of
Theorem~\ref{thm:existence and uniqueness}, let
$(Z,L,t_{\cE})$ denote the solution of \eqref{eq:TCRSDE}
with $T_t=t$. Then the solution for a general (possibly stochastic) time change $T$ is
\[
\bigl(Z\circ T,L\circ T,T^{-1}_{t_{\cE}-}\bigr),
\qquad
T^{-1}_u:=\inf\{s\geq0:T_s>u\}.
\]
\end{cor}

Classical results cover oblique reflection in bounded domains
\cite{LionsSznitman}. In unbounded domains, normal reflection 
with semimartingale drivers are analysed in~\cite{SlominskiLeszek1996} and
oblique reflection in generalized parabolic domains are studied in
\cite{Mijatovic1}. Time-change methods for SDEs without
reflection are developed in \cite{kobayashi2011}.
Here we establish strong existence and pathwise uniqueness
up to explosion for oblique reflection in regular unbounded
domains with driver $(T,W\circ T)$, together with the
representation above. In particular, we verify adaptedness
to the natural filtration of the time-changed driver.

Neither the existence theorem nor the representation requires
independence between $W$ and $T$. This allows dependent clocks,
provided the time-change hypotheses are satisfied.
For related asymptotic results with dependent clocks in one
dimension, see \cite{savov2020semi}.

\subsection{Proof strategy and organization}
Section~\ref{sec:asymptotics} proves our main results,
Theorem~\ref{thm:indep subordinator} and
Proposition~\ref{prop:indep subordinator}, taking
Theorem~\ref{thm:existence and uniqueness} and
Corollary~\ref{cor:rep of sol} as given.
A Lyapunov estimate yields uniform integrability, strengthening
the almost-sure growth law of
\cite[Theorem~2.2]{Mijatovic1} to $L^2$ convergence.
Independence then reduces the second-moment analysis to
inverse-subordinator moment estimates \cite{taqqu} and
Tauberian arguments \cite{bingham1987regular}.
The explosion assertions follow by applying the time-change
representation to the underlying reflected diffusion.

Section~\ref{sec:existence} establishes the supporting
existence, uniqueness, and representation results.
The proofs combine the bounded-domain theory of
\cite{LionsSznitman} with localization as in
\cite[Appendix~A]{Mijatovic1}, while also verifying
adaptedness to the natural filtration of the driving process.
Appendix~\ref{app:uniqueness} contains the It\^o calculation
used in the pathwise uniqueness proof.

\section{Proof of Theorem \ref{thm:indep subordinator}}
\label{sec:asymptotics}
We start by proving our main result, which is Theorem \ref{thm:indep subordinator}. We will assume Theorem \ref{thm:existence and uniqueness} and Corollary \ref{cor:rep of sol} to hold, and prove them later in the next section.

Our aim, here is to extend the results in \cite{Mijatovic1}, by establishing the superdiffusivity of the MSD of the solution $Z=(X,Y)$, to the non-time changed, driftless equation \eqref{eq:TCRSDE}. In particular notice that $X$ denotes the first coordinate of $Z$, living in $\R^+$ and $Y$ denotes the other $d$ coordinates, living in $\R^d$. The result is obtained mainly by using tools already introduced \cite{Mijatovic1, Mijatovic2}. Henceforth we assume \eqref{eq:domain-def},\eqref{eq:e1-projection}, \eqref{eq:Sigma-limit} and in particular the function $f$ in \eqref{eq:domain-def} is such that
\begin{equation}\label{eq:f growth}
    f(x)=o\parent{x^{\frac{3\beta-1}{2}}},\ f'(x)=o(x^{\frac{3\beta-3}{2}}),\ f''(x)=o(x^{\frac{3\beta-5}{2}})\text{ as }x\to\infty.
\end{equation}
More details on this model can be found in \cite{Mijatovic1}. We aim to prove the following result that shows that the superdiffusivity observed a.s. in the model in \cite{Mijatovic1}, is observed also in the MSD sense.
\begin{prop}\label{prop:L2 superdiffusivity}
    Let \eqref{eq:domain-def},\eqref{eq:e1-projection},\eqref{eq:Sigma-limit}, and \eqref{eq:f growth} hold, then for $\beta\in(-1,1)$ and any $z\in\cD$
    \begin{equation}\label{eq: L2 limit1}
        t^{-\frac{1}{1+\beta}}X_t\xrightarrow[]{L^2(\Omega,\Pb_z)} \kappa^{\frac{1}{1+\beta}},\quad t\to\infty
    \end{equation}
    where $\kappa:=\parent{\frac{s_0\sigma^2(\beta+1)}{2c_0a_{\infty}}}$. In particular $\E_z\norm{Z_t}^2_{d+1}\sim\E_z|X_t|^2\sim (\kappa t)^{\frac{2}{1+\beta}}$ as $t\to\infty$.
\end{prop}
To establish the behaviour of the second moment of the process we need full control on it, including its behaviour for small $t$. Thus as in \cite{Mijatovic2} we introduce the cut-off function\footnote{The function is slightly modified since the original had a discontinuous second derivative.}
$$M(x,v,u)=\text{exp}\parent{-\frac{(u-x)^3}{(u-v)^3\parent{(u-v)^2-(u-x)^2}}}\I_{(v,u]}(x)+\I_{(u,\infty)}(x)\quad u>v$$
For fixed $u>v$, denote by $M_{v,u}(x)$ the function $x\mapsto M(x,v,u)$ which is smooth  and such that
\begin{equation}\label{eq:properties of M}
\begin{split}
    &M_{v,u}(x)\geq0\ \forall\,x\geq0,\quad M_{v,u}(x)=0\ \forall\,x\leq v,\quad M_{v,u}(x)=1\ \forall\,x\geq u\\
    &\partial_xM_{v,u}(x)=\partial_{xx}^2M_{v,u}(x)=0\ \forall\,x\in[v,u]^c,\quad \partial_xM_{v,u}(x)\geq0\ \forall\,x\geq0
\end{split}
\end{equation}
We also consider the same Lyapunov function introduced by the authors in \cite{Mijatovic1}
$$g(z)=g(x,y):=x+\gamma\frac{\norm{y}^2_d}{b(x)},$$
and introduce $\nu: \partial\cD \to \R$ defined by $\nu (z)   := \langle \phi( z ) ,  \nabla g( z ) \rangle$, which will be useful later.\\
We state two technical lemmas which will be used in the proof of Proposition \ref{prop:L2 superdiffusivity}.
For a twice continuously differentiable scalar function $h$, write
\[
\Delta_\Sigma h(z)
:=\operatorname{tr}\bigl(\Sigma(z)D^2h(z)\bigr)
=\sum_{k,\ell=1}^{d+1}
  \Sigma_{k\ell}(z)\partial_{k\ell}h(z).
\]

\begin{lemma}\label{lemma:properties of g}
    Let \eqref{eq:domain-def},\eqref{eq:e1-projection},\eqref{eq:Sigma-limit}, and \eqref{eq:f growth} hold, then there exist $x'>0$ big enough, such that
    \begin{enumerate}
        \item There exists a constant $C>0$, independent of $x$, such that for all $y$ such that $(x,y)\in\cD$
        \begin{equation}\label{eq:pre-drift bound}
            \Delta_{\Sigma}g(x,y)\leq C(g(x,y))^{-\beta}\quad\forall x>x'
        \end{equation}
        \item For all $x>x'$ it holds 
        $$\inf_{y\in\partial\cD}\inn{\phi(x,y),e_x}\geq\frac{s_0}{2}>0$$
        \item For $\gamma>s_0/(2c_0)$ it holds
        \begin{equation}
            \inn{\nabla g(x,y),\phi(x,y)}\leq0\quad \forall\,(x,y)\in\partial\cD\cap([x',\infty)\times\R^d)
        \end{equation}
    \end{enumerate}
\end{lemma}
\begin{proof}
    Point (2) is just a consequence of \eqref{eq:e1-projection}, and point (3) is just \cite[Lemma 5.2]{Mijatovic1}. For (1) we have, from \cite[Lemma 5.1]{Mijatovic1}, that for any fixed $\varepsilon_1>0$ there exists $x_1=x_1(\varepsilon_1)$ such that
    \begin{equation}\label{eq:bDelta ineq}
        b(x)\Delta_{\Sigma}g(x,y)\leq C_1:= 2\gamma\sigma^2+\varepsilon_1\quad \forall\,(x,y)\in\cD\cap([x_1,\infty)\times\R^d)
    \end{equation}
    Moreover, from \cite[Lemma 5.1]{Mijatovic1}, we also know that as $x\to\infty$ $g(x,y)\sim x$ (uniformly in $y$) so also $g(x,y)^{\beta}\sim x^{\beta}$. Finally from \eqref{eq:domain-def}, $b(x)=a_{\infty}x^{\beta}+f(x)$ and $f(x)=o(x^{\beta})$, thus $b(x)\sim a_{\infty}x^{\beta}$, thus $g(x,y)^{\beta}\sim b(x)/a_{\infty}$, and for any $\varepsilon_2>0$ there exists $x_2$ sufficiently big, such that
    \begin{equation}\label{eq:g ineq}
         g(x,y)^{\beta}\leq \frac{1+\varepsilon_2}{a_{\infty}}b(x)\quad \forall\,(x,y)\in\cD\cap([x_2,\infty)\times\R^d)
    \end{equation}
    Let $\hat{x}=x_1\vee x_2$, combining \eqref{eq:bDelta ineq} and \eqref{eq:g ineq} (since we can assume $\Delta_{\Sigma}g\geq0$ for $x>x_1$) we get
    \begin{equation}
       g(x,y)^{\beta}\Delta_{\Sigma}g(x,y)\leq C\quad \forall\,(x,y)\in\cD\cap([\hat{x},\infty)\times\R^d)
    \end{equation}
    where $C:=C_1(1+\varepsilon_2)/a_{\infty}$. Since $g\geq0$, we can multiply both sides by $g(x,y)^{-\beta}$, and get \eqref{eq:pre-drift bound}. Then it suffices to take $x'\geq\hat{x}$ such that all the claims hold at once.
\end{proof}
Now let $x''>x'$, with $x'$ from Lemma \ref{lemma:properties of g}, and define a new function $G$
\begin{equation}\label{eq:Lyapunov G}
    G(x,y):=M_{x',x''}(x)g(x,y)+k(1-M_{x',x''}(x))\quad (x,y)\in\cD
\end{equation}
where $k$ is a constant such that $k=k(x',x'')>\sup_{z:x\in[x',x'']}|g(z)|$.
\begin{lemma}\label{lemma:properties of G}
    Under the same hypothesis of Lemma \ref{lemma:properties of g} we have that
    \begin{enumerate}
        \item There exists a constant $C>0$, independent of $x$, such that 
        \begin{equation}\label{eq:G drift bound}
            \Delta_{\Sigma}G(x,y)\leq C(G(x,y))^{-\beta}\quad\forall (x,y)\in\cD
        \end{equation}
        \item For $\gamma>s_0/(2c_0)$ it holds
        \begin{equation}\label{eq:G neg. local time}
            \inn{\nabla G(x,y),\phi(x,y)}\leq0\quad \forall\,(x,y)\in\partial\cD
        \end{equation}
        \item $\sup_{z\in\cD}\norm{\nabla G}_{d+1}<\infty$ and $G(x,y)\geq x$ for all $x\geq0$, and moreover $\inf_{z\in\cD}G(z)\geq\delta>0$.
    \end{enumerate}
\end{lemma}
\begin{proof}
    1) Notice that by definition
    \begin{equation}\label{eq:G shape}
        G(x,y):=\begin{cases}k\quad x<x'\\M_{x',x''}(x)g(x,y)+k(1-M_{x',x''}(x))\quad x'\leq x\leq x''\\ g(x,y)\quad x>x''\end{cases}
    \end{equation} 
    Therefore \eqref{eq:G drift bound} is trivial for $x<x'$, and because of  Lemma \ref{lemma:properties of g}, it holds for $x>x''$. On $[x',x'']$ the inequality is still verified since $G$ and $\Delta_{\Sigma}G$ are continuous functions and thus
    $$\sup_{\cD\cap([x',x'']\times\R^d)}\abs{(G(x,y))^{\beta}\Delta_{\Sigma}G(x,y)}\leq C$$
    which implies \eqref{eq:G drift bound} for all $x\in[x',x'']$.\\
    2) Notice that 
    $$\partial_xG(z)=\partial_xM_{x',x''}(x)(g(z)-k)+M_{x',x''}(x)\partial_xg(z)\quad \nabla_yG(z)=M_{x',x''}(x)\nabla_yg(z)$$
    Now, for $z\in\partial\cD$, consider the inner product with the vector field
    \begin{align*}
        \inn{\nabla G(z),\phi(z)}=&\partial_xM_{x',x''}(x)(g(z)-k)\inn{\phi(z),e_x}+M_{x',x''}(x)\inn{\nabla g(z),\phi(z)}
    \end{align*}
    By the properties in \eqref{eq:properties of M} we know that this is $0$ for $x\leq x'$, while for $x\geq x''$ is non-positive since $\inn{\nabla g(z),\phi(z)}\leq0$ by Lemma \ref{lemma:properties of g}. As for $x\in[x',x'']$ we have that $\partial_xM_{x',x''}(x)\geq0$ and by definition $g(z)-k\leq0$, while $\inn{\phi(z),e_x)}>0$, thus even on this interval we have non-positivity, and we can conclude that
    $$\inn{\nabla G(z),\phi(z)}\leq0\text{ for all }z\in\partial\cD$$
    3) Similarly to the previous point we can show the boundedness of the gradient, whereas looking at \eqref{eq:G shape}, it is clear that, for all $y$, $G$ is greater than $x''\wedge k$ for all $x\notin[x',x'']$, and on this interval is greater than $x$, since for $x=x'$  it holds $G(x',y)=k>0$, and for $x>x'$
    $$G(x,y)-x=M_{x',x''}(x)(g(x,y)-k)+k-x\geq (k-x)(1-M_{x',x''}(x))>0.$$
\end{proof}
Finally we need the following result, that extends Gronwall's lemma. Its proof can be found in \cite[Lemma 1]{LaSalle}.
\begin{lemma}[Bihari-LaSalle]\label{lemma:non-lin gronwall}
    Let $f(t)$ be non negative for $t\geq0$, $\Phi(t)$ non-negative and continuous for $t>0$ and such that 
    $$\int_0^t\Phi(s)ds\text{ exists for all }t\geq0$$
    and let $u(x)$ be non-negative, non-decreasing and continuous for $x\geq0$, and such that
    $$\int_{0^+}^x\frac{1}{u(y)}dy=\Psi(x) \text{ exists for all }x>0$$
    If it holds that
    $$f(t)\leq a+\int_0^t\Phi(s)u(f(s))ds\quad\forall\,t\geq0$$
    for some $a>0$, then
    $$\Psi(f(t))\leq\Psi(a)+\int_0^t\Phi(s)ds$$
\end{lemma}

We are now ready to prove Proposition~\ref{prop:L2 superdiffusivity}.

\begin{proof}[Proof of Proposition \ref{prop:L2 superdiffusivity}]
    Applying It\^o formula to $G(Z_t)$ yields
    $$G(Z_t)=G(z)+\int_0^t\inn{\nabla G(Z_s),\phi(Z_s)}dL_s+\frac{1}{2}\int_0^t\Delta_{\Sigma}G(Z_s)ds+\int_0^t\inn{\nabla G(Z_s),\sigma(Z_s)dW_s}$$
    In particular notice that
    \begin{align}\label{eq:qv G}
      \nonumber  [G(Z)]_t&=[M]_t=\int_0^t\sum_{i,j=1}^{d+1}\Sigma_{ij}(Z_s)\partial_jG(Z_s)\partial_iG(Z_s)ds\\&\leq\int_0^t\norm{\Sigma(Z_s)}_{op}\norm{\nabla G(Z_s)}^2_{d+1}ds\leq C't,
    \end{align}
    where $M$ is the martingale part of $G(Z)$, and $C'>0$ is a fixed deterministic constant, since $\Sigma=\sigma\sigma$ is bounded in operator norm and the gradient of $G$ is also bounded by Lemma \ref{lemma:properties of G}.
    Now let $\delta>0$ sufficiently small. Applying It\^o to $(G(Z_t))^{2+\delta}$ and using Lemma \ref{lemma:properties of G} and \eqref{eq:qv G}, we get
    \begin{align*}     G(Z_t)^{2+\delta}=&G(z)^{2+\delta}+\int_0^t(2+\delta)G(Z_s)^{1+\delta}\inn{\nabla G(Z_s),\phi(Z_s)}dL_s\\
    &+\int_0^t\frac{2+\delta}{2}G(Z_s)^{1+\delta}\Delta_{\Sigma}G(Z_s)ds+
        +\int_0^t\frac{(2+\delta)(1+\delta)}{2}G(Z_s)^{\delta} d[G(Z)]_s+M'_t\\
        \leq&G(z)^{2+\delta}+\int_0^t\parent{\frac{2+\delta}{2}CG(Z_s)^{1-\beta+\delta}+C'\frac{(2+\delta)(1+\delta)}{2}G(Z_s)^{\delta}}ds+M'_t
    \end{align*}
    Now, notice that for any $x>\delta'>0$, the function $v(x):=ax^{1-\beta+\delta}+bx^{\delta}$, is smaller then $w(x):=cx^{1-\beta+\delta}$, choosing an appropriate $c>0$. Thus, since by (3) in Lemma \ref{lemma:properties of G}, we can assume $G$ to be strictly positive (so greater or equal of a sufficiently small $\delta'$), then we have that for an appropriate $C''>0$, we get that 
    \begin{align}\label{eq:G inequality}
        G(Z_t)^{2+\delta}\leq&G(z)^{2+\delta}+\int_0^tC''G(Z_s)^{1-\beta+\delta}ds+M'_t
    \end{align}
    Now let $\rho_r:=\inf\graf{t>0:G(Z_t)^{2+2\delta}\geq r}$, and notice that using \eqref{eq:qv G}, we get
    $$[M']_{t\wedge\rho_r}=\int_0^{t\wedge\rho_r}(2+\delta)^2G(Z_s)^{2+2\delta}d[M]_s\leq (2+\delta)^2r[M]_{t\wedge\rho_r}\leq (2+\delta)^2rC't$$
    Therefore the process $M_{\rho_r\wedge\cdot}$ is actually a true martingale. Now notice that
    \begin{align}
        \int_0^{t}C''G(Z_{s\wedge\rho_r})^{1-\beta+\delta}ds=&\int_0^{t\wedge\rho_r}C''G(Z_{s})^{1-\beta+\delta}ds+\int_{\rho_r}^tG(Z_{\rho_r})^{1-\beta+\delta}ds\notag\\
        \geq&\int_0^{t\wedge\rho_r}C''G(Z_{s})^{1-\beta+\delta}ds\label{eq:t-rho_r}
    \end{align}
    Stopping \eqref{eq:G inequality} and using \eqref{eq:t-rho_r}, we get
    \begin{align}
        G(Z_{t\wedge\rho_r})^{2+\delta}\leq& G(z)^{2+\delta}+\int_0^{t}C''G(Z_{s\wedge\rho_r})^{1-\beta+\delta}ds+M'_{t\wedge\rho_r}\notag\\
        =& G(z)^{2+\delta}+\int_0^{t}C''\parent{G(Z_{s\wedge\rho_r})^{2+\delta}}^{\frac{1-\beta+\delta}{2+\delta}}ds+M'_{t\wedge\rho_r}\label{eq:stopped G inequality}
    \end{align}
    Now consider the function $u(x)=x^{\frac{1-\beta+\delta}{2+\delta}}$ and notice that since $\beta>-1$, $u(x)$ is concave, indeed
    $$\frac{1-\beta+\delta}{2+\delta}<1\iff \beta>-1$$
    Therefore we can take expectations in \eqref{eq:stopped G inequality} and get
    \begin{align*}
        \E_z[G(Z_{t\wedge\rho_r})^{2+\delta}]\leq&G(z)^{2+\delta}+\int_0^{t}C''\E_z\left[u\parent{G(Z_{s\wedge\rho_r})^{2+\delta}}\right]ds\\
        \leq&G(z)^{2+\delta}+\int_0^tC''u\parent{\E_z[G(Z_{s\wedge\rho_r})^{2+\delta}]}ds
    \end{align*}
    Thus we can apply Lemma \ref{lemma:non-lin gronwall} with $f(t):=\E_z[G(Z_{t\wedge\rho_r})^{2+\delta}]<\infty$, $\Phi(t)=C''$ and thus 
    $$\Psi(x)=\frac{2+\delta}{1+\beta}x^{\frac{1+\beta}{2+\delta}}$$
    and get 
    \begin{equation}\label{eq:A}
        \E_z[G(Z_{t\wedge\rho_r})^{2+\delta}]=f(t)\leq\parent{G(z)^{1+\beta}+C''t\frac{1+\beta}{2+\delta}}^{\frac{2+\delta}{1+\beta}}
    \end{equation}
    Finally, using Fatou's lemma, since $\rho_r\to\infty$ a.s.
    \begin{equation}\label{eq:upbound}
            \E_z\left[G(Z_{t})^{2+\delta}\right]\leq\liminf_{r\to\infty}\E_z[G(Z_{t\wedge\rho_r})^{2+\delta}]\leq\parent{G(z)^{1+\beta}+C''t\frac{1+\beta}{2+\delta}}^{\frac{2+\delta}{1+\beta}}\leq c_1+c_2t^{\frac{2+\delta}{1+\beta}}
    \end{equation}
    Where $c_1:=2^{\frac{2+\delta-1-\beta}{1+\beta}}G(z)^{2+\delta}$ and $c_2:=2^{\frac{2+\delta-1-\beta}{1+\beta}}\parent{C''\frac{1+\beta}{2+\delta}}^{\frac{2+\delta}{1+\beta}}$
    By point (3) in Lemma \ref{lemma:properties of G}, we have $X_t\leq G(Z_t)$ for all $t\geq0$, thus if we define $U_t\coloneqq t^{-\frac{2}{1+\beta}}X_t^2$, we get that
    \begin{equation}     \sup_{t\geq1}\E_zU_t^{1+\delta/2}\leq\sup_{t\geq1}\E_z\left[\frac{G(Z_t)^{2+\delta}}{t^{\frac{2+\delta}{1+\beta}}}\right]\leq c_1+c_2
    \end{equation}
    Which implies that the family $\graf{U_t:t\geq1}$ is uniformly integrable, which together with the a.s. convergence
    \begin{equation}\label{eq:as supdif}
        U_t\to \parent{s_0\sigma^2(\beta+1)/(2c_0a_{\infty})}^{2/(\beta+1)}
    \end{equation}
    given in \cite[Theorem 2.2]{Mijatovic1}, implies that $U_t$ converges in $L^1(\Omega,\Pb_z)$ to the constant given by $\parent{s_0\sigma^2(\beta+1)/(2c_0a_{\infty})}^{2/(\beta+1)}$, and this holds for all $z\in\cD$. This immediately implies the second moment convergence, since
    \begin{equation}\label{eq:moment convergence}
        \lim_{t\to\infty}\frac{\E_zX_t^2}{t^{\frac{2}{1+\beta}}}=\lim_{t\to\infty}\E_z\abs{U_t}=\parent{\frac{s_0\sigma^2(\beta+1)}{(2c_0a_{\infty})}}^{\frac{2}{\beta+1}}
    \end{equation}
    and the $L^2$-norm convergence, plus the a.s. convergence imply the $L^2$  convergence. Finally notice that
    \begin{equation}\label{eq:X and |Z|}
        X_t^2\leq\norm{Z_t}_{d+1}^2=X_t^2+\|Y_t\|_d^2\leq X_t^2+b(X_t)^2.
    \end{equation}
Since $b(x)=o(x)$ as $x\to\infty$ and $b$ is bounded on
compact intervals, for every $\varepsilon>0$ there exists
$C_\varepsilon<\infty$ such that
$b(x)^2\leq\varepsilon x^2+C_\varepsilon$ for all $x\geq0$.
By~\eqref{eq:X and |Z|} we obtain 
\[
0\leq
\frac{\E_z\norm{Z_t}_{d+1}^2}{\E_zX_t^2}-1
\leq \varepsilon+\frac{C_\varepsilon}{\E_zX_t^2}.
\]
Since $\E_zX_t^2\to\infty$, taking $t\to\infty$ and then
$\varepsilon\downarrow0$ yields
$\E_z\norm{Z_t}_{d+1}^2\sim\E_zX_t^2$.
\end{proof}

To deal with the proof of Theorem \ref{thm:indep subordinator}, it is useful to recall (see \cite{bingham1987regular}) that a function $f$ is said to be O-regularly varying ($OR$) at $0$ if it is positive, measurable and such that
\begin{equation}\label{eq:OR}
    0<\liminf_{x\to0^+}\frac{f(c x)}{f(x)}\leq\limsup_{x\to0^+}\frac{f(c x)}{f(x)}<\infty\quad c\geq1.
\end{equation}
We will also use the following lemma.
\begin{lemma}\label{lemma:asymptotics of the integral}
    Let $f,g:\R_+\to\R_+$ be measurable and locally bounded functions such that $g(t)\to\infty$ and $f\sim g$ at infinity; moreover let $(\mu_t)_{t\geq0}$ be a sequence of probability measures such that $\mu_t(K)\to0$ as $t\to\infty$ for $K$ compact. Let $F(t)=\int_0^{\infty}f(u)\mu_t(du)$, $G(t)=\int_0^{\infty}g(u)\mu_t(du)$ and assume additionally that $G(t)<\infty$ for all sufficiently
large $t$. Then $F\sim G$ as $t\to\infty$.
\end{lemma}
\begin{proof}
    First notice that $G(t)\to\infty$ as $t\to\infty$; indeed for any fixed $K>0$ there exists $V$ big enough such that $g(u)>2K$ for $u>V$, and a $S=S(V)$ big enough such that $\mu_t([0,V])<1/2$ for $t>S$, thus 
    $$G(t)=\int_0^Vg(u)\mu_t(du)+\int_V^{\infty}g(u)\mu_t(du)\geq2K\mu_t([V,\infty))\geq K$$
    Now fix $\varepsilon>0$, then let $U=U(\varepsilon)$ be such that $|f(u)-g(u)|<g(u)\varepsilon/2$ for $u>U$ and $T=T(U,\varepsilon)$ such that $\mu_t([0,U])<\varepsilon/(2M)$ for $t>T$, where $M=\sup_{[0,U]}|f(u)-g(u)|<\infty$ by the local boundedness assumption, then  
    $$\abs{F(t)-G(t)}\leq M\mu_t([0,U])+\int_U^{\infty}\abs{f(u)-g(u)}\mu_t(du)\leq\frac{\varepsilon}{2}+\frac{\varepsilon}{2}G(t).$$
    Now notice that by the first part we can assume $G(t)>1$, so that dividing by $G(t)$ the above we get 
    $$\abs{\frac{F(t)}{G(t)}-1}\leq \frac{\varepsilon}{2G(t)}+\frac{\varepsilon}{2}\leq\varepsilon.$$
\end{proof}
\begin{proof}[Proof of Proposition \ref{prop:indep subordinator} and Theorem \ref{thm:indep subordinator}]
    Let $m_z(t):=\E_zX_t^2$ and, by independence,
    $$M_z(t):=\E_z\TX^2_{t}=\E_zX_{T_t}^2=\int_0^{\infty}m_z(u)\Pb(T_t\in du),$$
    where $\TX$ is the first coordinate of $\TZ$, and the second inequality is justified by Corollary \ref{cor:rep of sol}. Notice that $m_z(u)\sim (\kappa u)^{\frac{2}{1+\beta}}$, where $m_z(u)$ is locally bounded by \eqref{eq:upbound} with $\delta=0$ and $\kappa$ is the constant in the RHS of \eqref{eq:moment convergence}. Moreover we have that $\mu_t(du):=\Pb(T_t\in du)$ satisfies the assumption of the lemma since
    $$\mu_t([0,U])=\Pb(T_t\leq U)=\Pb(S_U>t)\to0\quad\text{as }t\to\infty$$
    since we assume the subordinator to be unkilled. 
    
Applying Lemma \ref{lemma:asymptotics of the integral} to the function $U(t):=\E T_t^{\frac{2}{1+\beta}}$, requires integrability of $g(u)=(\kappa u)^{2/(1+\beta)}$ with respect to $\mu_t(du)$ for all large times $t$.
Since, for $\lambda>0$ and $u\geq0$, we have 
$\Pb(T_t>u)\leq\Pb(S_u\leq t)
\leq e^{\lambda t}\E e^{-\lambda S_u}
=e^{\lambda t-u\varphi(\lambda)}$,
for every $p>0$, we obtain
\[
\E T_t^p
=p\int_0^\infty u^{p-1}\Pb(T_t>u)\,du
\leq
\frac{e^{\lambda t}\Gamma(1+p)}{\varphi(\lambda)^p}
<\infty.
\]
Taking $p=2/(1+\beta)$ verifies the additional assumption
for $g(u)$.
    
Thus, by Lemma \ref{lemma:asymptotics of the integral} we get that 
    \begin{equation}\label{eq:inverse subordinator moment}
        M_z(t)\sim \kappa^{p}\int_0^{\infty}u^{p}\Pb(T_t\in du)=\kappa^p\E T_t^{p}=\kappa^p U(t).
    \end{equation}
    Now recall (see \cite{taqqu} formula (7)) that the Laplace transform of the $p$-moment of an inverse subordinator is given by\footnote{Note that \cite{taqqu} formula (7) contains a typo, and multiplies the formula by an extra $p$.}
    $$\Tilde{U}(\lambda)=\int_0^{\infty}e^{-\lambda t}U(t)dt=\frac{\Gamma(1+p)}{\lambda\varphi(\lambda)^{p}},$$
    thus the Laplace-Stieltjes transform is $\widehat{U}(\lambda)=\Gamma(1+p)\varphi(\lambda)^{-p}$.
    By \cite[Theorem 2.10.2]{bingham1987regular}, since $U$ is clearly non-decreasing, we have that if $\widehat{U}(\frac{1}{\cdot})$ is $OR$, then $\widehat{U}(\frac{1}{t})\asymp U(t)$ as $t\to\infty$. But notice that for $c\geq1$
    $$\frac{\widehat{U}(1/(c\lambda))}{\widehat{U}(1/(\lambda))}=\parent{\frac{\varphi(1/\lambda)}{\varphi(1/(c\lambda))}}^{p}=\parent{\frac{\varphi(cy)}{\varphi(y)}}^{p}$$
    where $y=1/(\lambda c)$. Now since $\varphi$ is a Bernstein function, it is increasing and sub-additive (see \cite[Corollary 3.8]{schilling2012bernstein}), thus if we set $\N\ni n=\lceil c\rceil\geq c$ we get $n\varphi(y)\geq\varphi(ny)\geq\varphi(cy)$, thus for all $y$
    $$1\leq\frac{\varphi(cy)}{\varphi(y)}\leq n.$$
    Taking the liminf and the limsup for $y\to0$ is equivalent to take them for $\lambda\to\infty$, thus we see that \eqref{eq:OR} is satisfied and $\widehat{U}(1/\cdot)$ is $OR$, and thus
    $$M_z(t)\asymp \kappa^pU(t)\asymp \kappa^p\widehat{U}(1/t)=\kappa^p\frac{\Gamma(1+p)}{(\varphi(1/t))^p}$$
    which proves Proposition \ref{prop:indep subordinator}.
    
    To prove Theorem \ref{thm:indep subordinator}, it is sufficient to notice that if $\varphi(t)\sim ct^{\alpha}\ell(t)$, with $\ell$ slowly varying at zero, then $\widehat{U}(1/t)$ is such that
    $$\widehat{U}(1/t)=\Gamma(1+p)(\varphi(1/t))^{-p}\sim\Gamma(1+p)c^{-p}(1/t)^{-\alpha p}(\ell(1/t))^{-p}\quad t\to\infty.$$
    It follows from \cite[Theorem 1.7.1]{bingham1987regular} that
    $$U(t)\sim\frac{1}{\Gamma(1+\alpha p)}\widehat{U}(1/t).$$
    Thus, combining this with \eqref{eq:inverse subordinator moment}, we get
    $$M_z(t)\sim \frac{\kappa^p\Gamma(1+p)c^{-p}}{\Gamma(1+\alpha p)}t^{\alpha p}(\ell(1/t))^{-p}.$$
    With the same reasoning of equation \eqref{eq:X and |Z|}, we conclude that the same asymptotic holds for $\E_z\norm{\TZ_t}_{d+1}^2$.\\
    Now we just need to cover the proof of the first part of Theorem \ref{thm:indep subordinator}, about explosion. Since $S$ is finite and $S_{\infty}:=\lim_{t\to\infty}S_t=\infty$ a.s., we have that 
    $$\TE=S_{\taue-}<\infty\iff\taue<\infty\iff B(\infty)<\infty\iff\beta<-1$$
    Now assume $\beta<-1$, then
    \begin{align*}
        \E_z[\TE]=&\int_0^{\infty}\E[S_{u-}]\Pb_z(\taue\in du)=\E[S_1]\int_0^{\infty}u\Pb_z(\taue\in du)=\E[S_1]\E_z[\taue],
    \end{align*}
    and the above is finite if and only if $\E[S_1]<\infty$. In particular, we used independence and the fact that by monotone convergence and properties of Lévy processes
    $$\E[S_{t-}]=\E[\lim_{\varepsilon\to0}S_{t-\varepsilon}]=\lim_{\varepsilon\to0}\E[S_{t-\varepsilon}]=t\E[S_1].$$  
 \end{proof}

\section{Proof of Theorem \ref{thm:existence and uniqueness}}
\label{sec:existence}
First of all we establish some technical, but rather mild assumptions on the domain $\cD$ and the vector field $\phi$. We require that there exists an increasing sequence $(\cD^{(n)})_{n}$ of $\cC^3$ and bounded domains, and on each of them we want to be able to put a $\cC^2_b$ vector field $\phi^{(n)}$ satisfying
\begin{equation}\label{eq:vector field}
    \inf_{x\in\partial\cD^{(n)}}\inn{\phi^{(n)}(x),n^{(n)}(x)}>0,
\end{equation}
where $n^{(n)}(x)$ denotes the inward pointing unit normal at $x\in\partial \cD^{(n)}$. In particular we require that the sequence $(\cD^{(n)})_n$ is approximating $\cD$, i.e. $\cD=\cup_n\cD^{(n)}$, and that for any $r>0$ there exists $N_r$ such that
\begin{equation}\label{eq:domain exhaustion1}
     \cD\cap B_r=\cD^{(n)}\cap B_r,\quad B_r\cap\partial\cD=B_r\cap\partial\cD^{(n)},\quad\phi^{(n)}(z)=\phi(z)\quad z\in B_r\cap\partial\cD^{(n)},
\end{equation}
for all $n\geq N_r$, where $B_r$ is the closed ball of radius $r$. 
\begin{rem}\label{rem:domain example}
    These assumptions might seem a bit technical, but they are actually quite broad and exclude only domains with pathological behaviors far from the origin. Indeed we are just asking that the domain can be approximated by bounded domains obtained by intersecting the original with a ball and then adding a "cap" to smooth out the angles, and that on their boundary we can extend the vector field smoothly.

    For generalized parabolic domains, the same localization by bounded
domains and compatible reflection fields is used in
\cite[proof of Thm~A.1]{Mijatovic1}. Their approximating domains
are taken to be $\cC^2$; under our $\cC^3$ boundary assumption,
the caps can be chosen $\cC^3$, with $\cC^2$ inward-pointing
reflection fields agreeing with $\phi$ on the unchanged boundary.
    \begin{figure}[ht!]
        \centering
        \includegraphics[width=0.9\linewidth]{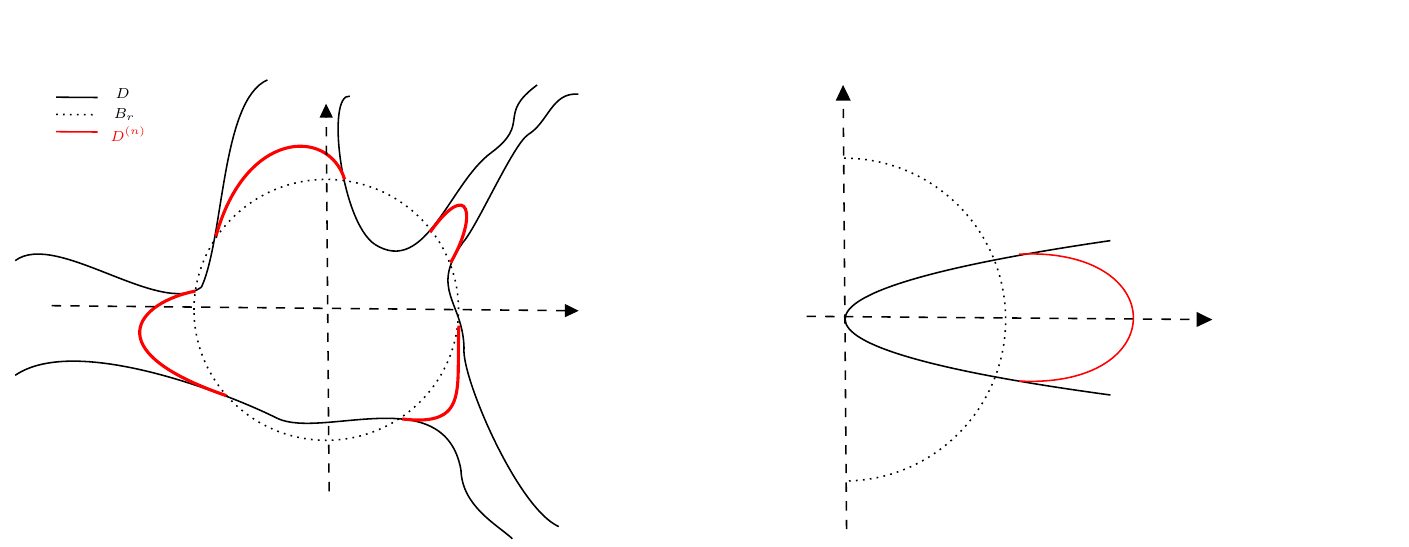}
        \caption{\textit{On the left a generic $\cC^3(\R^2)$ domain, with one of the approximating sets $\cD^{(n)}$. On the right a parabolic domain, as the ones treated in \cite{Mijatovic1}.}}
        \label{fig:1}
    \end{figure}
\end{rem}

We will treat \eqref{eq:TCRSDE} in the general context of SDEs driven by continuous semimartingales with reflection, which have the general form
\begin{equation}\label{eq:semimartingale RSDE}
    X_t=x+\int_0^tf(X_{s})dZ_s+\int_0^t\phi(X_s)dK_s,\quad K_t=\int_0^t\I(X_s\in\partial\cD)dK_s.
\end{equation}
Here $Z$ is a semimartingale with respect to its natural filtration\footnote{All the filtration considered in the paper are assumed to be augmented with the null sets and right-continuous, so that they satisfy the usual conditions.} and has continuous paths and $f$ is a càdlàg function. A strong solution to \eqref{eq:semimartingale RSDE} is a pair $(X,K)$ of continuous processes defined on the same probability space of $Z$, adapted to $\cF^Z$, and such that $\Pb(X\in\overline{\cD})=1$ and $K$ is non-decreasing, thus of bounded variation, and they satisfy almost surely \eqref{eq:semimartingale RSDE}. In particular, a solution to \eqref{eq:semimartingale RSDE} is pathwise unique if any two solutions $(X,K),(X',K')$ defined on the same probability space satisfy 
$$\Pb\parent{(X_t,K_t)=(X_t',K_t'),\ \forall\,t\geq0}=1$$
First of all we need to establish whether the driving process in \eqref{eq:TCRSDE} is a semimartingale with respect to its natural filtration. To answer this question we will need the following.
\begin{thm}[Stricker's Theorem, \cite{protter2005stochastic} Theorem II.2.4]\label{prop:Stricker+}
    Let $\cH$ be a filtration and $\cG$ a sub-filtration; moreover let $Z$ be a continuous $\cH$-semimartingale, which is also adapted to $\cG$ then $Z$ is a $\cG$-semimartingale.
\end{thm}
Now we precise the definition of time change and of stopped filtration, which will be useful for the sequel.
\begin{defi}
    Given probability space $(\Omega,\cF,\Pb)$, and a filtration $\cH=(\cH_t)_{t\geq0}$, we say that a process $T=(T_t)_{t\geq0}$ is an $\cH$-time change, if it has càdlàg, non-decreasing  trajectories, and is such that $T_0=0$ a.s. and each $T_t$ is an $\cH$-stopping time; in particular, we say it is finite if $T_t<\infty$ a.s. for all $t\geq0$. 
\end{defi}
\begin{rem}\label{rem:stopped filtration}
    Recall that, given a filtration $\cH$, and an $\cH$ time-change $T$, we can define the stopped filtration $\cG=(\cG_t)$ where
    $$\cG_t:=\cH_{T_t}=\graf{A:A\cap\graf{T_t\leq s}\in\cH_s\ \forall\,s\geq0},$$
    In particular, the process $T$ is $\cG$-adapted and given an $\cH$-semimartingale $Z$, then $Z\circ T$ it is a $\cG$-semimartingale (see \cite[Proposition V.1.5]{revuz2013continuous} or \cite[Lemma 2.2]{kobayashi2011}). The natural question we want to answer is whether $(Z_{T_t})_t$ is a semimartingale also with respect to its natural filtration. The following lemma gives an affirmative answer, and thus guarantees that the question of a strong solution to an SDE driven by a time-changed semimartingale is actually well-posed.
\end{rem}
\begin{lemma}\label{lemma:Stricker2}
    Let $\cH=(\cH_t)_{t\geq0}$ be a filtration such that $Z$ is a $\cH$-semimartingale and $T$ is a finite $\cH$-time change, then $Z\circ T$ is a semimartingale with respect to its natural filtration.
\end{lemma}
\begin{proof}
    From the above remark $Z\circ T$ is a $(\cH_{T_t})$-semimartingale. Since $Z$ is $\cH$-adapted and $T_s$ is an $\cH$-stopping time, we have $Z_{T_s}\in\cH_{T_s}\subset\cH_{T_t}$ for all $s\leq t$; therefore the $\sigma$-algebra generated by the family $(Z_{T_s})_{s\leq t}$, must be contained in $\cH_{T_t}$, i.e. $\cF^{Z\circ T}_t=\sigma(Z_{T_s};s\leq t)\subseteq\cH_{T_t}$. Therefore $\cF^{Z\circ T}$ is a sub-sigma-algebra of $(\cH_{T_t})$ to which $Z\circ T$ is adapted, and thus by Stricker's theorem (see Proposition \ref{prop:Stricker+}), $Z\circ T$ must be a an $\cF^{Z\circ T}$-semimartingale.
\end{proof}
Hence \eqref{eq:TCRSDE} can be rewritten as \eqref{eq:semimartingale RSDE}, by considering $f=(b,\sigma^1,...,\sigma^d)$ and $Z=(T,W\circ T)$ which, is a continuous semimartingale, with respect to $\cF^{W\circ T}$. Indeed,  by Lemma \ref{lemma:Stricker2}, we have that $W\circ T$ is a $\cF^{W\circ T}$ semimartingale, and $T_t=[W^1\circ T]_t\in\cF^{W\circ T}_t$ for all $t\geq0$, where $W^1$ is the first coordinate of $W$. More specifically we have the following.
\begin{lemma}\label{lemma:local martingale}
    If $(W,\cH)$ is a Brownian motion and $T$ is a continuous finite $\cH$-time change, then $W\circ T$ is a $(\cH_{T_t})$-continuous local martingale.
\end{lemma}
\begin{proof}
    Let $\tau_n\coloneqq\inf\graf{t>0:|W_t|>n}$ and define the localizing sequence for the time-changed process $\sigma_n\coloneqq\inf\graf{t>0:|W_{T_t}|>n}$, then we have that $T_{t\wedge\sigma_n}=T_t\wedge\tau_n$, and thus
    $$(W\circ T)^{\sigma_n}=W_{T_{\cdot\wedge\sigma_n}}=W_{T_{\cdot}\wedge\tau_n}=(W^{\tau_n})\circ T.$$
    Integrability and adaptedness to $\cH_{T_t}$ for $(W\circ T)^{\sigma_n}$ are immediate. The martingale property is recovered by Doob's stopping theorem since $W^{\tau_n}$ is a bounded $\cH$-martingale and $T_s\leq T_t$ are finite $\cH$-stopping times, thus $$\E[W_{T_{t\wedge\sigma_n}}|\cH_{T_s}]=\E[(W^{\tau_n})_{T_t}|\cH_{T_s}]=(W^{\tau_n})_{T_s}=W_{T_s\wedge\sigma_n}.$$
\end{proof}
The next theorem will establish strong existence and pathwise uniqueness for \eqref{eq:TCRSDE}, in the setting presented in the introduction, but when $\cD$ is a bounded domain, thus we still don't need to deal with explosion. For the proof, we will use some known facts, which we state here for convenience.
\begin{lemma}[\cite{LionsSznitman},Lemma 4.1, Remark 3.1]\label{lemma:oblique}
Let $\cD$ be a bounded $\cC^3$ open set in $\R^d$, $n(x)$ the inward normal vector at $x$, and $\phi\in\cC_b(\R^d,\R^d)$ such that \ref{eq:vector field} holds.
\begin{enumerate}
    \item There exists a symmetric matrix valued function $A:\R^d\to\R^{d\times d}$, such that
    \begin{equation}\label{eq:nu}
        A(x)\geq\nu'I_d,\quad \nu'>0\quad A(x)\phi(x)=n(x)\ \forall\,x\in\partial\cD
    \end{equation}
    In particular there exists $C_0\geq0$ such that
    $$C_0\norm{x-x'}^2_d-\inn{A(x)\phi(x),(x-x')}\geq0\quad x\in\partial\cD,\ x'\in\overline{\cD}$$
    In addition, if $\phi\in\cC_b^2(\R^d)$ then $A_{ij}\in\cC_b^2(\R^d)$.
    \item If $\phi\in\cC_b^2(\R^d)$, then there exists $\Phi\in\cC_b^2(\R^d)$ such that there exists $\alpha>0$ for which
    \begin{equation}\label{eq:alpha}
        \inn{\nabla\Phi(x),\phi(x)}\geq\alpha>0\quad\forall\,x\in\partial\cD
    \end{equation} 
\end{enumerate}
\end{lemma} 

\begin{rem}
    Notice that this lemma is the reason why we require the domain to be $\cC^3$. The uniqueness argument for oblique reflection exploits the existence of a matrix valued function $A$, that transforms the oblique vector field in a normal vector field. We need $A$ to be $\cC^2(\R^d)$ because we need to apply It\^o's formula. Such regularity generally requires the boundary of the domain to be $\cC^3$. For detail, one can check Lemma 4.1 in \cite{LionsSznitman} and the subsequent discussion. 
\end{rem}

\begin{prop}[\cite{revuz2013continuous} Proposition V.1.4, V.1.5]\label{prop:change of var}
Let $(T_t)_t$ be a continuous finite $\cH$-time change with $T_0=0$, and let $\cG=(\cH_{T_t}):$
    \begin{enumerate}
        \item Suppose $H$ is $\cH$-progressively measurable, then $H\circ T$ is $\cG$-progressively measurable. Moreover if $V$ is of finite variation we have
        $$\int_0^{T_t}H_sdV_s=\int_0^tH_{T_s}dV_{T_s}$$
        \item Suppose that $X$ is an $\cH$-semimartingale, and $H\in L(X,\cH)$, then $(H_{T_{t-}})_{t\geq0}\in L(X\circ T,\cG)$ and a.s.
        $$\int_0^{T_t}H_sdX_s=\int_0^tH_{T_{s-}}dX_{T_s}$$
    \end{enumerate}
\end{prop}

We will also need the following lemma.
\begin{lemma}\label{lemma:measurability}
    Let $W$ be an $\cH$-Brownian motion and $T$ a continuous $\cH$-time change, then the random element satisfies $(T_t,(W_{s\wedge T_t})_{s\geq0})\in\cF^{W\circ T}_t$.
\end{lemma}
\begin{proof}
    Write $\widetilde W=W\circ T$ and, for $u\geq0$, define
    $H_u:=\inf\{s\geq0:T_s\geq u\}$.
    Since $T_t=[\widetilde W^1]_t$ (where $\widetilde W^1$ denotes the first component of $\widetilde W$), the process $T$ is adapted
    to $\cF^{\widetilde W}$. By continuity and monotonicity of $T$, we obtain
    $\{H_u\leq t\}=\{T_t\geq u\}$, so $H_u$ is an
    $\cF^{\widetilde W}$-stopping time. Moreover, for any fixed $u\geq0$, we have
    \[
    W_{u\wedge T_t}=\widetilde W_{H_u\wedge t}
    \in\cF_t^{\widetilde W}.
    \]
    Taking rational $u\geq0$ and using continuity proves the
    claimed measurability of the stopped Brownian path.
\end{proof}

So we are finally ready to state and prove our existence and uniqueness result in bounded domains.
\begin{thm}\label{thm:ex-uniq-bounded}
    Let $\cD$ be a bounded $\cC^3$ domain in $\R^d$, $\phi\in\cC^2_b$, $\inf_{x\in\partial\mathcal D}\langle\phi(x),n(x)\rangle>0$, $\mu$ and $\sigma$ bounded and Lipschitz. Then \eqref{eq:TCRSDE} admits a pathwise unique strong solution $(\TZ,\TL)$. In particular, if we denote by $(Z,L)$ the solution to \eqref{eq:TCRSDE} when $T_t=t$, then the solution for a general $T$ can be written explicitly as $(Z\circ T,L\circ T)$.
\end{thm}
\begin{proof}
    \underline{Uniqueness}. Let $(\TZ,\TL),(\TZ',\TL')$ be two solutions to \eqref{eq:TCRSDE}. Let $A$ and $\Phi$ be as in Lemma \ref{lemma:oblique} and define
    \begin{equation}\label{eq:gamma}
        \gamma(t)=\text{exp}\graf{-\lambda(\Phi(\TZ_t)+\Phi(\TZ_t'))}(\TZ_t-\TZ_t')^T(A(\TZ_t)+A(\TZ_t'))(\TZ_t-\TZ_t').
    \end{equation}
    Then, by slightly generalizing the first inequality in \cite[Proposition 4.1]{LionsSznitman} we obtain
    \begin{equation}\label{eq:uniqueness inequality}
        \gamma(t)\leq M_t+C_1\int_0^t|\TZ_s-\TZ_s'|^2\text{exp}\graf{-\lambda(\Phi(\TZ_s)+\Phi(\TZ_s')}dT_s,
    \end{equation}
    where $M$ is a local martingale, and $\lambda,C_1>0$ are two fixed constants. 
    The proof of \eqref{eq:uniqueness inequality} is just a long and tedious application of the It\^o formula, and can be found it in the appendix at the end of the paper.
    
    Now, define $\eta(t)=\gamma(t)e^{-\beta T_t}$, which is such that $\eta(t)\geq0$ and $\eta(0)=\gamma(0)=0$, thus if we show that for an appropriate choice of $\beta$, $\eta$ it is a local supermartingale, then $\eta(t)=0$ for all $t\geq0$ a.s. and therefore $\TZ$ and $\TZ'$ are indistinguishable. By applying the product rule and using \eqref{eq:uniqueness inequality} we get
    \begin{align*}
        \eta(t)=&\int_0^te^{-\beta T_s}d\gamma(s)-\beta\int_0^te^{-\beta T_s}\gamma(s)dT_s\\
        \leq&\int_0^te^{-\beta T_s}dM_s+\\
        \quad&+\int_0^te^{-\beta T_s}\parent{C_1|\TZ_s-\TZ_s'|^2\text{exp}\graf{-\lambda(\Phi(\TZ_s)+\Phi(\TZ_s')}-\beta\gamma(s)}dT_s
    \end{align*}
    Now notice that by the properties of $A(x)$ in Lemma \ref{lemma:oblique}, we have
    $$(\TZ_t-\TZ_t')^T(A(\TZ_t)+A(\TZ_t'))(\TZ_t-\TZ_t')\geq2\nu'|\TZ_t-\TZ_t'|^2$$
    Therefore
    \begin{align*}
        C_1|\TZ_s-\TZ_s'|^2&\text{exp}\graf{-\lambda(\Phi(\TZ_s)+\Phi(\TZ_s')}-\beta\gamma(s)\\
        &\leq |\TZ_s-\TZ_s'|^2\text{exp}\graf{-\lambda(\Phi(\TZ_s)+\Phi(\TZ_s')}(C_1-2\nu'\beta)
    \end{align*}
    So choosing $\beta>C_1/(2\nu')$ we get a negative drift term, thus
    $$\eta(t)\leq\int_0^te^{-\beta T_s}dM_s=N_t$$
    Clearly $N$ is still a continuous local martingale, thus there exists $(\tau_n)_n$ such that $\tau_n\uparrow\infty$
    $$0\leq\E[\eta(t)]\leq\liminf_{n\to\infty}\E[\eta(t\wedge\tau_n)]\leq\liminf_{n\to\infty}\E [N_{t\wedge\tau_n}]=N_0=0\quad\forall\,t\geq0$$
    Therefore $\eta(t)=0$ a.s. for all $t\geq0$, and then using the continuity of $\eta$ we conclude that $\eta\equiv0$ a.s., which implies that $\TZ$ and $\TZ'$ are indistinguishable. Now, by the indistinguishability of $\TZ,\TZ'$
    $$\int_0^t\phi(\TZ_s)d(\TL-\TL')_s=0,\quad \forall\,t\geq0,\ a.s.$$
    Moreover, by \eqref{eq:vector field} we have $\phi(x)\neq0$ for all $x\in\partial\cD$ we must have $\cL_t=c+\cL_t'$ for all $t\geq0$ a.s. In particular $c=0$ since $\cL_0=\cL'_0=0$.

    \smallskip
    \noindent \underline{Existence}. Consider any probability space, let $\cH$ be a filtration such that $T$ is a $\cH$-time change and $W$ is an $\cH$-Brownian motion. Let $T_t=t$, then according to \cite{LionsSznitman}, Equation \eqref{eq:TCRSDE} admits a strong solution $(Z,L)$, which is adapted to $\cF^W$, and thus, also to the larger filtration $\cH$.
    
    Now by taking an arbitrary $\cH$-time change $T$, we see that, the processes $(Z\circ T,L\circ T)$ satisfy
    \begin{equation}\label{eq:solution}
    \begin{split}
        Z_{T_t} =& z+\int_0^{T_t}\mu(Z_{s})ds+\int_0^{T_t}\sigma(Z_{s})dW_{s}+\int_0^{T_t}\phi(Z_{s})dL_s\\
        =&z+\int_0^t\mu(Z_{T_s})dT_s+\int_0^t\sigma(Z_{T_s})dW_{T_s}+\int_0^t\phi(Z_{T_s})dL_{T_s},\\
        L_{T_t}=&\int_0^{T_t}\I\graf{Z_s\in\partial\cD}dL_s=\int_0^{t}\I\graf{Z_{T_s}\in\partial\cD}dL_{T_s},
    \end{split}
    \end{equation}
    where we used change of variable formulas in Proposition \ref{prop:change of var}. Notice that, by Remark \ref{rem:stopped filtration}, $(Z\circ T,L\circ T)$ are $(\cH_{T_t})$-adapted, but $\cF^{W\circ T}\subset \cH_{T_{\cdot}}$, as shown in the proof of Lemma \ref{lemma:Stricker2}, so to confirm that $(Z\circ T,L\circ T)$ is a strong solution, we need to verify that they are adapted to $\cF^{W\circ T}$.
    
    First of all, notice that the fact that $(Z_t,L_t)\in\cF_t^W$ for every $t\geq0$, implies (see for example \cite[Lemma 1.14]{kallenberg2021foundations}) that for each fixed $t\geq0$, there exists a Borel measurable function $h_t:\cC([0,\infty),\R^d)\to\R^d\times\R_+$, such that
    \begin{equation}\label{eq:a.s.}
        (Z_t,L_t)=h_t((W_{t\wedge s})_{s\geq0})\quad a.s.
    \end{equation}
    Now define
    $$T^{(n)}_t=\frac{\lfloor 2^nT_t\rfloor}{2^n}\in 2^{-n}\N,\ \text{which is such that }T_t^{(n)}\leq T_t,\ T_t^{(n)}\uparrow T_t.$$
    Now notice that, almost surely (in particular everywhere except on the countable union, over $k$, of null-sets where \eqref{eq:a.s.} does not hold), it holds
    $$(Z_{T^{(n)}_t},L_{T^{(n)}_t})=\sum_{k=0}^{\infty}\I(T_t^{(n)}=k/2^n)h_{k/2^n}\parent{\parent{W_{s\wedge k/2^n}}_{s\geq0}}\in\cF_t^{W\circ T}.$$
    Where the measurability follows by Lemma \ref{lemma:measurability}, and the fact that $T^{(n)}_t\leq T_t$, and that products and sums of measurable function is still measurable. Finally by the continuity of $(Z,L)$ we can take the limit which is still measurable,
    $$(Z_{T_t},L_{T_t})=\limn(Z_{T^{(n)}_t},L_{T^{(n)}_t})\in\cF^{W\circ T}_t$$
    and this holds for each $t\geq0$.
\end{proof}

\begin{rem}\label{rem:existence and uniqueness}
    To prove existence we show that the time change of the solution to the reflection problem without time-change is a solution. Therefore this works whenever the problem without time-change admits a strong solution, even with weaker assumptions on the domain, the coefficients and the vector field.  
\end{rem}

Since we are allowing oblique reflection in an unbounded domain, the solution to \eqref{eq:TCRSDE} might explode. The procedure developed in \cite{Mijatovic1}, can be easily extended, applying some obvious modifications, so that it holds in a more general setting. Equip $\cD\cup\{\partial\}$ with the one-point
compactification topology. Thus, for sequences in $\cD$,
$z_n\to\partial$ if and only if $|z_n|\to\infty$. We define $\cE(f):=\inf\graf{t\in\R_+:f(t)=\partial}$, where $\partial$ is the state accounting for explosion, and we consider the space
$$\overline{\cC}:=\graf{f\in\cC(\R_+,\cD\cup\graf{\partial}):f(t)=\partial,\ \forall\,t\geq\cE(f)}$$
endowed with the compact-open topology. Such space can be interpreted as the set of all possible explosive trajectories. It is easy to show that
$$\cE(f)=\lim_{r\to\infty}S_r(f),\quad S_r(f):=\inf\graf{t\in\R_+:\abs{f(t)}\geq r}.$$
The definition of strong solution up to explosion time given in \cite[Appendix A]{Mijatovic1}, still works for \eqref{eq:TCRSDE}. Thus, to prove Theorem \ref{thm:existence and uniqueness}, we proceed in three steps: first we show that there exists unique a couple $(\TZ,\TL)$, such that for any $r>0$
\begin{equation}\label{eq:stopped_TCRSDE}
\begin{split}
    \TZ_{t\wedge\sigma_r}=& z+\int_0^{t\wedge\sigma_r} \mu(\TZ_{s})dT_s+\int_0^{t\wedge\sigma_r}\sigma(\TZ_{s})dW_{T_s}+\int_0^{t\wedge\sigma_r}\phi(\TZ_s)d\TL_s,\\
    \TL_{t\wedge\sigma_r}=&\int_0^{t\wedge\sigma_r}\I(\TZ_s\in\partial\cD)d\TL_s, \quad\text{where }\sigma_r:=S_r(\TZ). 
\end{split}
\end{equation}
Then to show that $\TZ$ has trajectories a.s. in $\overline{\cC}$ we consider first the case $T_t=t$, and finally apply a change of variable to extend to the general case.
\begin{proof}[Proof of Theorem \ref{thm:existence and uniqueness} and Corollary \ref{cor:rep of sol}]
    (I) Let $(\cD^{(n)})_n$, $(\phi^{(n)})_n$ as in \eqref{eq:domain exhaustion1}, and $\cD_r:=B_r\cap\cD, \cD_r^{(n)}:=B_r\cap\cD^{(n)}$. 
    Consider the reflection problem of $\cD^{(n)}$ with vector field $\varphi^{(n)}$. By Theorem \ref{thm:ex-uniq-bounded}
    \begin{equation}\label{eq:bounded_TCRSDE}
    \begin{split}
        \TZ_{t}^{(n)}=& z+\int_0^{t} \mu(\TZ_{s}^{(n)})dT_s+\int_0^{t}\sigma(\TZ_{s}^{(n)})dW_{T_s}+\int_0^{t}\phi^{(n)}(\TZ_s^{(n)})d\TL_s^{(n)},\\
        \TL_{t}^{(n)}=&\int_0^{t}\I(\TZ_s^{(n)}\in\partial\cD^{(n)})d\TL_s^{(n)}, 
    \end{split}
    \end{equation}
    admits a unique solution $(\TZ^{(n)},\TL^{(n)})$; in particular, let $\sigma^{(k)}_r:=S_r(\TZ^{(k)})$ and notice that for any $n\geq N_r$ we have that $(\TZ^{(n)}_{\cdot\wedge\sigma_r^{(n)}},\TL^{(n)}_{\cdot\wedge\sigma_r^{(n)}})$ satisfies
    \begin{equation}\label{eq:bounded stopped TCRSDE}
    \begin{split}
        \TZ_{t\wedge\sigma_r^{(n)}}^{(n)}=& z+\int_0^{t\wedge\sigma_r^{(n)}} \mu(\TZ_{s}^{(n)})dT_s+\int_0^{t\wedge\sigma_r^{(n)}}\sigma(\TZ_{s}^{(n)})dW_{T_s}+\int_0^{t\wedge\sigma_r^{(n)}}\phi(\TZ_s^{(n)})d\TL_s^{(n)},\\
        \TL_{t\wedge\sigma_r^{(n)}}^{(n)}=&\int_0^{t\wedge\sigma_r^{(n)}}\I(\TZ_s^{(n)}\in\partial\cD)d\TL_s^{(n)}, 
    \end{split}
    \end{equation}
    where we are using the fact that $\phi^{(n)}(z)=\phi(z)$ and $\I(z\in\partial\cD^{(n)})=\I(z\in\partial\cD)$ for $z\in B_r\cap\partial \cD^{(n)}$ by \eqref{eq:domain-def}, and $\TZ_{s}^{(n)}\in\cD_r$ for $s\leq\sigma_r^{(n)}$. By uniqueness of the solution to \eqref{eq:bounded_TCRSDE}, it follows the uniqueness for \eqref{eq:bounded stopped TCRSDE} and thus 
    \begin{equation}\label{eq:stopped solutions}
        \parent{\TZ^{(n)}_{\cdot\wedge\sigma_r},\TL^{(n)}_{\cdot\wedge\sigma_r}}=\parent{\TZ^{(m)}_{\cdot\wedge\sigma_r},\TL^{(m)}_{\cdot\wedge\sigma_r}},\quad \forall\,m\geq n\geq N_r,
    \end{equation}
    where $\sigma_r:=\lim_{n\to\infty}\sigma^{(n)}_r$, and we have $\sigma_r=\sigma^{(n)}_r$ for all $n\geq N_r$. In other words, the solutions relative to two successive approximation of the domain are exactly the same when stopped upon exiting a ball small enough, such that $\cD_r^{(n)}=\cD_r^{(m)}=\cD_r$. Thus, we can define $\cT_{\cE}:=\lim_{r\to\infty}\sigma_r$ (notice that $\sigma_r$ is a stopping time with respect to the natural filtration of $\TZ$, thus $\cT_{\cE}$ is also a stopping time with respect to the same filtration) and
    \begin{equation}\label{eq:limit stopped}
        \TZ_t:=\begin{cases}\lim_{r\to\infty}\TZ^{(N_r)}_{t\wedge\sigma_r},\quad t<\cT_{\cE}\\\partial\quad t\geq\cT_{\cE}\end{cases}\qquad\TL_t:=\lim_{r\to\infty}\TL_{t\wedge\sigma_r}^{(N_r)},\quad t<\cT_{\cE}
    \end{equation}
    moreover, since $\TL_t$ is non-decreasing, we can define $\TL_{\cT_{\cE}}:=\lim_{t\to\cT_{\cE}}\TL_t$ and $\TL_t=\TL_{\cT_{\cE}}$ for $t\geq\cT_{\cE}$. The limits \eqref{eq:limit stopped} are well defined since the sequence is eventually constant by \eqref{eq:stopped solutions}. In particular, let $t<\cT_{\cE}$, then there exists $r(t)$ such that $t<\sigma_r$ for all $r\geq r(t)$, and one has $\TZ_{t}=\TZ^{(n)}_{t}$ and $\TL_{t}=\TL^{(n)}_{t}$, for all $n\geq N_{r}$. Thus, by stopping we have 
    \begin{equation}\label{eq:stopped equality}
        \TZ_{t\wedge\sigma_r}=\TZ^{(n)}_{t\wedge\sigma_r}\qquad\TL_{t\wedge\sigma_r}=\TL^{(n)}_{t\wedge\sigma_r} \quad\forall\,n\geq N_r\text{ and }r>0,
    \end{equation}
    thus upon substituting in \eqref{eq:bounded stopped TCRSDE} we have that $(\TZ,\TL)$ satisfies \eqref{eq:stopped_TCRSDE} and it is unique. In particular, it is clear from the construction and from Theorem \ref{thm:ex-uniq-bounded} that $(\TZ,\TL)=(Z\circ T,L\circ T)$, where $(Z,L)$ are the processes obtained for $T_t=t$. Moreover we have that  
    $$\sigma_r=\inf\graf{t\geq0:|\TZ_t|\geq r}=\inf\graf{t\geq0:T_t=\hat{\sigma}_r}=T^{-1}_{\hat{\sigma}_r-},$$
    where $\hat{\sigma}_r$ denotes the exit time for $Z$ and  $T^{-1}_t:=\inf\graf{s\geq0:T_s>t}$ is the right-inverse of $T$, which is right-continuous. Therefore, taking limits we see that the explosion time can be expressed as $\cT_{\cE}=T^{-1}_{\taue-}$.\\
    (II) Now assume $T_t=t$, and denote $(\TZ,\TL,\cT_{\cE})=(Z,L,\taue)$. To show that the trajectories of $Z$ are in $\overline{\cC}$ we just need to show continuity at the explosion time, i.e. $Z_t\to\partial$ as $t\uparrow\taue$. Notice that such limit holds if and only if $\lim_{t\to\taue}\abs{Z_t}^2=\infty$. By It\^o formula we have
    $$\abs{Z_t}^2=|z|^2+2\int_0^{t}\inn{Z_s,\mu(Z_s)}ds+2\int_0^t\inn{Z_s,\sigma(Z_s)dW_s}+2\int_0^t\inn{Z_s,\phi(Z_s)}dL_s+\int_0^t\trace\Sigma(Z_s)ds.$$
    Consider the submartingale $(V_{t\wedge\sigma_r})_{t\geq0}$, defined by
    \begin{equation}\label{eq:V}
        V_{t\wedge\sigma_r}=|z|^2+m_{t\wedge\sigma_r}+2\int_0^{t\wedge\sigma_r}\parent{\inn{Z_s,\mu(Z_s)}^++\trace\Sigma(Z_s)}ds+2\int_0^{t\wedge\sigma_r}\inn{Z_s,\phi(Z_s)}^+dL_s,
    \end{equation}
    where $m_{t\wedge\sigma_r}:=2\int_0^{t\wedge\sigma_r}\inn{Z_s,\sigma(Z_s)dW_s}$ is a true martingale since
    $$\E[m]_{t\wedge\sigma_r}\leq 4\int_0^{t\wedge\sigma_r}Z_s^{\tra}\Sigma(Z_s)Z_sds\leq4r^2\norm{\Sigma}_{op}t.$$
    In particular, $|Z_{t\wedge\sigma_r}|^2\leq V_{t\wedge\sigma_r}$, and using Doob's maximal inequality we get that
    \begin{align}
        \Pb_z(\sigma_{2r}\leq t)=&\Pb_z\parent{\sup_{0\leq s\leq t\wedge\sigma_{2r}}|Z_s|^2\geq 4r^2}\notag
        \leq\Pb_z\parent{\sup_{0\leq s\leq t\wedge\sigma_{2r}}V_s\geq 4r^2}\notag\\
        \leq&\frac{1}{4r^2}\E_zV_{t\wedge\sigma_{2r}}\label{eq:V ineq1}
    \end{align}
    Now notice that by \eqref{eq:V} and the boundedness of $\mu$ and $\phi$ we have
    \begin{align}
        \E_zV_{t\wedge\sigma_{2r}}\leq&|z|^2+4rC(t+\E_zL_{t\wedge\sigma_{2r}})\label{eq:V ineq2}
    \end{align}
    We can show that for $t\downarrow0$ the expectation of the local time goes to 0. Indeed, notice that by \eqref{eq:stopped equality} we have $L_{t\wedge\sigma_{2r}}=L^{(m)}_{t\wedge\sigma_{2r}}$ for $m$ sufficiently big, where $L^{(m)}$ is the local time relative to the bounded domain $\cD^{(m)}$. Let $N:\R^d\to\R$ be a function coinciding with the signed distance from $\partial\cD^{(m)}$, in a tubular neighborhood of such set, and recall $\cD_{2r}\cap\partial\cD^{(m)}=\cD_{2r}\cap\partial\cD$; it is known (e.g. \cite{gilbarg2013elliptic}) that $N$ can be taken in $\cC^2(\R^d)$ and it is such that $\nabla N(z)=n(z)$ for $z\in\cD_{2r}\cap\partial\cD$. Now by It\^o formula we get
    \begin{align*}
        N(Z_{t\wedge\sigma_{2r}})-N(z)&=M_{t\wedge\sigma_{2r}}+\int_0^{t\wedge\sigma_{2r}}\parent{\inn{\nabla N(Z_s),\mu(Z_s)}+\frac{1}{2}\trace(\Sigma(Z_s)H_N(Z_s))}ds+\\
        &\quad+\int_0^{t\wedge\sigma_{2r}}\inn{n(Z_s),\phi(Z_s)}dL_s\\
        &\geq M_{t\wedge\sigma_{2r}}-C_1t+\alpha L_{t\wedge\sigma_{2r}},
    \end{align*}
    where $\alpha:=\inf_{z\in\partial\cD}\inn{\phi(z),n(z)}>0$ by \eqref{eq:vector field}, and we used that 
    \begin{align*}
        \sup_{z\in\cD_{2r}}\big|\inn{\nabla N(z),\mu(z)}+&\frac{1}{2}\trace(\Sigma(Z_s)H_N(Z_s))\big|\\ &\leq \sup_{z\in\cD_{2r}}\parent{\abs{\nabla N(z)}\abs{\mu(z)}+\frac{1}{2}\trace(\Sigma(z)H_N(z))}=:C_1<\infty.
    \end{align*}
    Notice that $M_{\cdot\wedge\sigma_{2r}}$ is a true martingale since 
    $$\E[M]_{\cdot\wedge\sigma_{2r}}\leq 2\E\int_0^{t\wedge\sigma_{2r}}\nabla N(Z_s)^{\tra}\Sigma(Z_s)\nabla N(Z_s)ds\leq C_2t,$$
    thus taking expectation we get the following
    $$\E L_{t\wedge\sigma_{2r}}\leq\alpha^{-1}\parent{\E N(Z_{t\wedge\sigma_{2r}})-N(z)+C_1t}.$$
To make the convergence uniform in the starting point, for arbitrary $\varepsilon>0$, replace $N$
by a smooth truncation $N_\varepsilon$ of the signed distance,
with $\nabla N_\varepsilon=n$ on the boundary and
$\sup N_\varepsilon-\inf N_\varepsilon\leq C\varepsilon$.
The same It\^o argument gives
\[
\sup_{z\in\cD_r}\E_z L_{t\wedge\sigma_{2r}}
\leq \alpha^{-1}\bigl(C\varepsilon+C_{r,\varepsilon}t\bigr),
\]
where the constants are independent of $z$.
Hence, since $\varepsilon>0$ was arbitrary, we obtain $$\sup_{z\in\cD_r}\E_z L_{t\wedge\sigma_{2r}}\to0 \quad\text{as $t\to0$.}$$
Therefore one can take $t_0=t_0(r)$ sufficiently small, such that $\E_zL_{t\wedge\sigma_{2r}}\leq1/C$ for all $z\in\cD_r$ and $t$ sufficiently small. Set $t_1=t_0\wedge C^{-1}$, so that
for $r$ sufficiently large~\eqref{eq:V ineq1} and \eqref{eq:V ineq2} imply
    \begin{equation*}
        \Pb_z(\sigma_{2r}\leq t_1)\leq\frac{3}{4}\implies\inf_{z\in\cD_r}\Pb_z(\sigma_{2r}\geq t_1)\geq\frac{1}{4}.
    \end{equation*}
  
    One can then use the same Borel-Cantelli argument used in \cite{Mijatovic1} to conclude that $|Z_t|^2\to\infty$ as $t\to\taue$. This proves that the trajectories of $Z$ are in $\overline{\cC}$. We also show that on $\graf{\taue<\infty}$ the explosion happens because of the local time, i.e. $L_t\to\infty$ as $t\to\taue$, while the remaining part does not explode. To do this we just notice that
    \begin{align*}
        \abs{Z_t}^2\leq&4\parent{\abs{z}^2+\abs{\int_0^t\mu(Z_s)ds}^2+|\cN_t|^2+\abs{\int_0^t\phi(Z_s)dL_s}^2}\\
        \leq&4\abs{z}^2+Ct^2+4|N_t|^2+4\abs{\int_0^t\phi(Z_s)dL_s}^2,
    \end{align*}
    where we use the boundedness of $\mu$ and where $\cN_t:=\int_0^t\sigma(Z_s)dW_s$, which is a true martingale by the boundedness of $\sigma$. As $t\to\taue$ on $\graf{\taue<\infty}$, we see that the only two components that can explode in the right hand side are either $|\cN|^2$ or the local time integral. Using the same localizing argument as in \cite{Mijatovic1} we show that $|\cN_{\taue}|^2<\infty$  on $\graf{\taue<\infty}$, thus by the boundedness of $\phi$ we conclude that $L_{\taue}=\lim_{t\to\taue}L_t=\infty$.
    \\
    (III) For the general case we use that $(\TZ,\TL)=(Z\circ T,L\circ T)$ from point (I) and thus the continuity and the limits in \eqref{eq:limit}, immediately follow from point (II) and the continuity of $T$. 
\end{proof}
\appendix
\section{An estimate for pathwise uniqueness}
\label{app:uniqueness}
\begin{lemma}\label{lemma:appendix}
     Let $\cD$ be a bounded $\cC^3$ domain in $\R^d$, $\phi\in\cC^2_b$ with $\inf_{x\in\partial\cD}\inn{\phi(x),n(x)}>0$, $\mu$ and $\sigma$ bounded and Lipschitz, and let $A,\Phi,\nu',\alpha,C_0$ be as in Lemma \ref{lemma:oblique}. Let $(\TZ,\TL),(\TZ',\TL')$ be two solutions to \eqref{eq:TCRSDE} defined on the same probability space and with the same starting point $z$, and let $\gamma$ be defined as in \eqref{eq:gamma}. Then there exist constants $\lambda,C_1>0$, a continuous local martingale $M$ and a continuous non-decreasing process $K$, with $M_0=K_0=0$, such that
    \begin{equation}\label{eq:appendix}
        \gamma(t)= M_t+C_1\int_0^t|\TZ_s-\TZ_s'|^2\exp\graf{-\lambda(\Phi(\TZ_s)+\Phi(\TZ_s'))}dT_s-K_t,\qquad t\geq0.
    \end{equation}
    In particular, $\gamma(t)\leq M_t+C_1\int_0^t|\TZ_s-\TZ_s'|^2\exp\graf{-\lambda(\Phi(\TZ_s)+\Phi(\TZ_s'))}dT_s$.
\end{lemma}
\begin{proof}
    Denote $\widetilde W=W\circ T$. Since $[\widetilde W^k,\widetilde W^\ell]_t=\delta_{kl}T_t$, all quadratic
    covariation terms below are contracted against $dT_t$.
    For a matrix-valued function $A$, 
    we apply $\Delta_\Sigma$ entrywise:
    $(\Delta_\Sigma A)_{ij}:=\Delta_\Sigma(A_{ij})=\sum_{k,\ell=1}^d
    \Sigma_{k\ell}\,\partial_{k\ell}A_{ij}$.
    
    Denote $\gamma(t)=E_tQ_t$ where $E_t=e^{-\lambda V_t}$ and $Q_t=D_t^TA_tD_t$ and
    $$V_t=\Phi(\TZ_t)+\Phi(\TZ_t'),\quad D_t=\TZ_t-\TZ_t'\quad A_t=A(\TZ_t)+A(\TZ_t').$$
    Recall that
    \begin{align*}
    d\TZ_t=\mu(\TZ_t)dT_t+\sigma(\TZ_t)d\widetilde W_t+\phi(\TZ_t)d\TL_t\quad d\TZ_t'=\mu(\TZ_t')dT_t+\sigma(\TZ_t')d\widetilde W_t+\phi(\TZ_t')d\TL'_t,
    \end{align*}
    and thus
    \begin{align*}
    dD_t=(\mu(\TZ_t)-\mu(\TZ_t'))dT_t+(\sigma(\TZ_t)-\sigma(\TZ_t'))d\widetilde W_t+\phi(\TZ_t)d\TL_t-\phi(\TZ_t')d\TL'_t.
    \end{align*}
    Therefore by It\^o formula applied entrywise to $A(\TZ_t)$ and $A(\TZ_t')$, using that $\TL,\TL'$ have finite variation, we get
    \begin{align*}
    dA_t=&\parent{\nabla A(\TZ_t)\mu(\TZ_t)+\nabla A(\TZ_t')\mu(\TZ_t')}dT_t\\
    &+\frac{1}{2}\bigl((\Delta_\Sigma A)(\TZ_t)
     +(\Delta_\Sigma A)(\TZ'_t)\bigr)dT_t\\
    &+\inn{\nabla A(\TZ_t),\sigma(\TZ_t)d\widetilde W_t}+\inn{\nabla A(\TZ_t'),\sigma(\TZ_t')d\widetilde W_t}+\\
    &+\nabla A(\TZ_t)\phi(\TZ_t)d\TL_t+\nabla A(\TZ_t')\phi(\TZ_t')d\TL'_t,
    \end{align*}
    where we recall that $\Sigma(z)=\sigma(z)\sigma(z)$. By the product rule (the covariation terms only involve the $d\widetilde W$-parts of $D$ and $A$, hence are absolutely continuous with respect to $dT_t$) it holds that
    \begin{align*}
    dQ_t=&(dD_t^T)A_tD_t+D_t^T(dA_t)D_t+D_t^TA_t(dD_t)+d[D^T,A]_tD_t+d[D^TA,D]_t\\
    =&\quad\Bigg(D_t^T \left( \nabla A(\TZ_t)\mu(\TZ_t) + \nabla A(\TZ'_t)\mu(\TZ'_t)\right) D_t+2 D_t^T A_t (\mu(\TZ_t) - \mu(\TZ'_t))\Bigg)dT_t+ \\
    &+\Bigg(\frac{1}{2}D_t^T
\bigl((\Delta_\Sigma A)(\TZ_t)
     +(\Delta_\Sigma A)(\TZ'_t)\bigr)D_t+ \text{tr}\left(A_t (\sigma(\TZ_t) - \sigma(\TZ'_t))(\sigma(\TZ_t) - \sigma(\TZ'_t))^T\right) +\\
    &\qquad +2 \sum_{i,j=1}^d D_t^i \left[ (\sigma(\TZ_t) - \sigma(\TZ'_t)) \left( \sigma(\TZ_t)^T \nabla A_{ij}(\TZ_t) + \sigma(\TZ'_t)^T \nabla A_{ij}(\TZ'_t) \right) \right]_j \Bigg)dT_t+\\
    &+ \Bigg[ D_t^T \left( \langle \nabla A(\TZ_t), \sigma(\TZ_t) d\widetilde W_t \rangle + \langle \nabla A(\TZ'_t), \sigma(\TZ'_t) d\widetilde W_t \rangle \right) D_t + 2 D_t^T A_t (\sigma(\TZ_t) - \sigma(\TZ'_t)) d\widetilde W_t \Bigg] \\
    &+ \Bigg[ D_t^T (\nabla A(\TZ_t) \phi(\TZ_t)) D_t + 2 D_t^T A_t \phi(\TZ_t) \Bigg] d\TL_t \\
    &+ \Bigg[ D_t^T (\nabla A(\TZ'_t) \phi(\TZ'_t)) D_t - 2 D_t^T A_t \phi(\TZ'_t) \Bigg] d\TL'_t.
    \end{align*}
    Now notice that using the boundedness of $A,\nabla A, D^2A,\mu,\sigma$ and the Lipschitz continuity of $\mu,\sigma$ we get that the terms in front of $dT_t$ are  bounded by $c_1|D_t|^2$ for some $c_1>0$. We can apply again It\^o formula, this time on $V$
    \begin{align*}
    dV_t=&\parent{\inn{\nabla\Phi(\TZ_t),\mu(\TZ_t)}+\inn{\nabla\Phi(\TZ_t'),\mu(\TZ_t')}}dT_t+\\
    &+\frac{1}{2}\parent{\Delta_\Sigma\Phi(\TZ_t)+\Delta_\Sigma\Phi(\TZ_t')}dT_t+\\
    &+\inn{\nabla\Phi(\TZ_t),\sigma(\TZ_t)d\widetilde W_t}+\inn{\nabla\Phi(\TZ_t'),\sigma(\TZ_t')d\widetilde W_t}+\\
    &+\inn{\nabla\Phi(\TZ_t),\phi(\TZ_t)}d\TL_t+\inn{\nabla\Phi(\TZ_t'),\phi(\TZ_t')}d\TL'_t,
    \end{align*}
    and we can easily obtain its quadratic variation process as
    \begin{align*}
    d[V]_t=&\Big(\nabla\Phi(\TZ_t)^T\Sigma(\TZ_t)\nabla\Phi(\TZ_t)+\nabla\Phi(\TZ_t')^T\Sigma(\TZ_t')\nabla\Phi(\TZ_t')\\
    &+2\nabla\Phi(\TZ_t)^T\sigma(\TZ_t)\sigma(\TZ_t')^T\nabla\Phi(\TZ_t')\Big)dT_t.
    \end{align*}
    A further application of It\^o brings the following expression for $E$ 
    \begin{align*}
    dE_t=&-\lambda E_tdV_t+\frac{\lambda^2}{2}E_td[V]_t\\
    =&-\lambda E_t\Big(\inn{\nabla\Phi(\TZ_t),\mu(\TZ_t)}+\inn{\nabla\Phi(\TZ_t'),\mu(\TZ_t')}\Big)dT_t\\
    &+E_t\Bigg(-\frac{\lambda}{2}\parent{\Delta_\Sigma\Phi(\TZ_t)+\Delta_\Sigma\Phi(\TZ_t')}+\\
    &\ \qquad+\frac{\lambda^2}{2}\big(\nabla\Phi(\TZ_t)^T\Sigma(\TZ_t)\nabla\Phi(\TZ_t)+\nabla\Phi(\TZ_t')^T\Sigma(\TZ_t')\nabla\Phi(\TZ_t')\\
    &\qquad\ +2\nabla\Phi(\TZ_t)^T\sigma(\TZ_t)\sigma(\TZ_t')^T\nabla\Phi(\TZ_t')\big)\Bigg)dT_t\\
    &-\lambda E_t\parent{\inn{\nabla\Phi(\TZ_t),\sigma(\TZ_t)d\widetilde W_t}+\inn{\nabla\Phi(\TZ_t'),\sigma(\TZ_t')d\widetilde W_t}}+\\
    &-\lambda E_t\parent{\inn{\nabla\Phi(\TZ_t),\phi(\TZ_t)}d\TL_t+\inn{\nabla\Phi(\TZ_t'),\phi(\TZ_t')}d\TL'_t}.
    \end{align*}
    Using the boundedness of $\Phi,\nabla \Phi,D^2\Phi,\mu,\sigma$ we get that the terms in front of $dT_t$ are bounded by $E_tc_2$ for some constant $c_2$.
    Applying the product rule to $\gamma(t)$ we get 
    $$d\gamma(t)=Q_tdE_t+E_tdQ_t+d[E,Q]_t.$$
    In particular using the boundedness and Lipschitz hypothesis we get that $d[E,Q]_t=\rho_t\,dT_t$ with $|\rho_t|\leq c_3E_t|D_t|^2$.
    Collecting the terms, we can write
    \begin{equation*}
        d\gamma(t)=dM_t+\psi_t\,dT_t+\chi_t\,d\TL_t+\chi'_t\,d\TL'_t,
    \end{equation*}
    where $M$ is the continuous local martingale, with $M_0=0$, collecting all the $d\widetilde W$-integrals coming from $dE_t$ and $dQ_t$, and $\psi,\chi,\chi'$ are the coefficients of $dT,d\TL,d\TL'$ respectively.
    Then using the other bounds stated above, and the fact that $Q_s\leq 2\norm{A}_\infty|D_s|^2$, we get that there exists some $C_1>0$ such that 
    \begin{equation*}
        \psi_t\leq C_1E_t|D_t|^2\quad\text{for all }t\geq0.
    \end{equation*}
    Now the terms driven by $\TL,\TL'$ are respectively
     $$\chi_t=-\lambda E_tQ_t\inn{\nabla\Phi(\TZ_t),\phi(\TZ_t)}+ E_t\parent{D_t^T (\nabla A(\TZ_t) \phi(\TZ_t)) D_t + 2 D_t^T(A(\TZ_t)+A(\TZ_t'))\phi(\TZ_t)},$$
    $$\chi'_t=-\lambda E_tQ_t\inn{\nabla\Phi(\TZ_t'),\phi(\TZ_t')}+ E_t\parent{D_t^T (\nabla A(\TZ_t') \phi(\TZ_t')) D_t - 2 D_t^T(A(\TZ_t)+A(\TZ_t'))\phi(\TZ_t')}.$$
    Since $d\TL_t$ is carried by $\{t:\TZ_t\in\partial\cD\}$ and $d\TL'_t$ by $\{t:\TZ'_t\in\partial\cD\}$, it suffices to bound $\chi_t$ when $\TZ_t\in\partial\cD$ and $\chi'_t$ when $\TZ'_t\in\partial\cD$; in particular we may apply Lemma \ref{lemma:oblique} with $x=\TZ_t\in\partial\cD$, $x'=\TZ'_t\in\cD$ (respectively with the roles exchanged).
    Using Lemma \ref{lemma:oblique} and $Q_t\geq2\nu'|D_t|^2$ we get that $-\lambda E_tQ_t\inn{\nabla\Phi(\TZ_t),\phi(\TZ_t)}\leq-2\lambda\alpha\nu'E_t|D_t|^2$, and also
    \begin{align*}
        D_t^T(A(\TZ_t)+A(\TZ_t'))\phi(\TZ_t)=&(\TZ_t-\TZ_t')^T(2A(\TZ_t)+(A(\TZ_t')-A(\TZ_t)))\phi(\TZ_t)\\
        =&2(\TZ_t-\TZ_t')^TA(\TZ_t)\phi(\TZ_t)+(\TZ_t-\TZ_t')^T(A(\TZ_t')-A(\TZ_t))\phi(\TZ_t)\\
        \leq&2C_0|D_t|^2+k|D_t|^2,
    \end{align*}
    \begin{align*}
        -D_t^T(A(\TZ_t)+A(\TZ_t'))\phi(\TZ_t')=&-(\TZ_t-\TZ_t')^T(2A(\TZ_t')+(A(\TZ_t)-A(\TZ_t')))\phi(\TZ_t')\\
        =&2(\TZ_t'-\TZ_t)^TA(\TZ_t')\phi(\TZ_t')+(\TZ_t'-\TZ_t)^T(A(\TZ_t)-A(\TZ_t'))\phi(\TZ_t')\\
        \leq&2C_0|D_t|^2+k|D_t|^2.
    \end{align*}
    Where we used that $\phi$ is bounded, \eqref{eq:nu}, and the fact that $A\in\cC_b^2$ thus it is Lipschitz continuous. Using again the boundedness of $\nabla A,\phi$ we get $D_t^T (\nabla A(\TZ_t) \phi(\TZ_t)) D_t \leq c|D_t|^2$.
    Thus combining we get, on the supports of $d\TL$ and $d\TL'$ respectively,
    \begin{equation*}
        \chi_t\leq E_t|D_t|^2\big(c+4C_0+2k-2\lambda\nu'\alpha\big),\qquad \chi'_t\leq E_t|D_t|^2\big(c+4C_0+2k-2\lambda\nu'\alpha\big).
    \end{equation*}
    So choosing $\lambda\geq\frac{4C_0+2k+c}{2\nu'\alpha}$ we get $\chi_t\leq0$ on the support of $d\TL$ and $\chi'_t\leq0$ on the support of $d\TL'$. Finally, define
    \begin{equation*}
        K_t:=\int_0^t\parent{C_1E_s|D_s|^2-\psi_s}dT_s-\int_0^t\chi_s\,d\TL_s-\int_0^t\chi'_s\,d\TL'_s.
    \end{equation*}
    All three integrands are non-negative on the support of the corresponding integrator, hence $K$ is continuous, non-decreasing and $K_0=0$. Since $\gamma(0)=0$, integrating the expression for $d\gamma$ gives exactly \eqref{eq:appendix}.
\end{proof}

\section*{Funding}
AM was also supported  the EPSRC grants EP/V009478/1 and EP/W006227/1. FV acknowledges financial support under the National Recovery and Resilience Plan (NRRP), Mission 4, Component 2, Investment 1.1, Call for tender No. 104 published on 2.2.2022 by the Italian Ministry of University and Research (MUR), funded by the European Union – NextGenerationEU– Project Title “Non–Markovian Dynamics and Non-local Equations” – 202277N5H9 - CUP:D53D23005670006 - Grant Assignment Decree No. 973 adopted on June 30, 2023, by the Italian Ministry of University and Research (MUR). FV was partially supported by INDAM (GNAMPA group).

\section*{AI statement}
The main ideas underlying the proofs of the principal results were developed by the authors without AI assistance. During the writing process, AI tools (GPT-5.6 Sol and Claude Sonnet 5) were used for checking technical arguments, identifying relevant references, and correcting spelling and language. 

\bibliographystyle{alpha}

\bibliography{References}

\end{document}